\documentclass[twoside,12pt]{amsart}
\usepackage{amssymb}
\usepackage{fullpage}

\numberwithin{equation}{section}

\newtheorem{Proposition}[equation]{Proposition}
\newtheorem{Lemma}[equation]{Lemma}
\newtheorem{Theorem}[equation]{Theorem}
\newtheorem{Corollary}[equation]{Corollary}

\theoremstyle{definition}
\newtheorem{Remark}[equation]{Remark}

\def\Tab{{\operatorname{Tab}}}
\def\Row{{\operatorname{Row}}}
\def\sTab{{\operatorname{sTab}}}
\def\Irr{{\operatorname{Irr}}}

\def\iso{\cong}
\def\into{\hookrightarrow}
\def\onto{\twoheadrightarrow}
\def\isoto{\stackrel{\sim}{\to}}
\def\sub{\subseteq}
\def\lan{\langle}
\def\ran{\rangle}
\def\bar{\overline}
\def\ad{{\operatorname{ad}\,}}
\def\ne{{\operatorname{ne}}}

\def\Ann{\operatorname{Ann}}
\def\RS{\operatorname{RS}}
\def\Wh{\operatorname{Wh}}

\def\C{{\mathbb C}}
\def\R{{\mathbb R}}
\def\Z{{\mathbb Z}}

\def\bp{{\mathbf p}}

\def\op{\operatorname{op}}
\def\Prim{\operatorname{Prim}}
\def\col{\operatorname{col}}
\def\row{\operatorname{row}}
\def\End{{\operatorname{End}}}

\def\GL{\operatorname{GL}}
\def\Sp{\operatorname{Sp}}
\def\SO{\operatorname{SO}}

\def\a{\mathfrak a}
\def\b{\mathfrak b}
\def\g{\mathfrak g}
\def\h{\mathfrak h}
\def\k{\mathfrak k}
\def\l{\mathfrak l}
\def\m{\mathfrak m}
\def\n{\mathfrak n}
\def\p{\mathfrak p}
\def\q{\mathfrak q}
\def\s{\mathfrak s}
\def\t{\mathfrak t}
\def\gl{\mathfrak{gl}}
\def\sl{\mathfrak{sl}}
\def\so{\mathfrak{so}}
\def\sp{\mathfrak{sp}}

\def\cO{\mathcal O}
\def\cL{\mathcal L}
\def\cS{\mathcal S}
\def\cX{\mathcal X}

\def\word{{\operatorname{word}}}
\def\eps{\epsilon}

\title{On changing highest weight theories for finite $W$-algebras}
\author{Jonathan Brown and Simon M.~Goodwin}

\address{School of Mathematics, University of Birmingham, Birmingham, B15 3LX,~UK}
\email{brownjs@for.mat.bham.ac.uk, goodwin@for.mat.bham.ac.uk}


\begin{document}

\begin{abstract}
A highest weight theory for a finite $W$-algebra $U(\g,e)$ was developed in \cite{BGK}.
This leads to a strategy for classifying the irreducible finite dimensional
$U(\g,e)$-modules.
The highest weight theory depends
on the choice of a parabolic subalgebra of $\g$ leading to
different parameterizations of the finite dimensional irreducible $U(\g,e)$-modules.
We explain how to construct an isomorphism preserving bijection between the
parameterizing sets
for different choices of parabolic subalgebra when $\g$ is of type
$A$, or when $\g$ is of types C or D and $e$ is an even multiplicity
nilpotent element.
\end{abstract}

\maketitle

\section{Introduction}

Let $U(\g,e)$ be the finite $W$-algebra associated to the nilpotent
element $e$ in a reductive Lie algebra $\g$ over $\C$.  Finite
$W$-algebras were introduced to the mathematical literature by
Premet in \cite{Pr1} and have subsequently attracted a lot of
interest, see for example the recent survey \cite{Lo4}.  In
\cite{BGK} a highest weight theory for $U(\g,e)$ is developed. The
key theorem required for the highest weight theory, \cite[Theorem
4.3]{BGK}, says that there is a subquotient of $U(\g,e)$ isomorphic
to $U(\g_0,e)$, where $\g_0$ is a minimal Levi subalgebra of $\g$
containing $e$. This allows a definition of Verma modules by
inducing finite dimensional irreducible $U(\g_0,e)$-modules. These
Verma modules have irreducible heads and all finite dimensional
irreducible $U(\g,e)$-modules can be realized in this manner.

At present the classification of finite dimensional irreducible $U(\g,e)$-modules
is unknown, except in some special cases.
In \cite{BK} Brundan and Kleshchev
classified these modules in the case that
$\g$ is of type A.  In \cite{Bro2} the first author found the classification in
the case that $\g$ is classical and $e$ is a {\em rectangular} nilpotent.
In \cite{BroG} the authors classified the finite dimensional
irreducible $U(\g,e)$-modules with integral central character
in the case that $\g$ is classical and $e$ is an {\em even multiplicity} nilpotent.
All of these classifications can be stated nicely in terms of the highest weight theory.

One particular feature of this highest weight theory is that it  requires the
choice of $\q$, a parabolic subalgebra of $\g$ which contains $\g_0$ as a Levi subalgebra.
For a finite dimensional irreducible $U(\g_0,e)$-module $V$,
we denote the Verma module corresponding to $V$ and $\q$ by
$M(V,\q)$ and write $L(V,\q)$ for its irreducible head.
Let $\q, \q'$ be two parabolic subalgebras of $\g$ containing $\g_0$ as a Levi subalgebra,
and let $V$, $V'$ be two finite dimensional irreducible $U(\g_0,e)$-modules.
It is a natural to ask: when is $L(V,\q) \iso L(V',\q')$?
The main purpose of this note is to answer this question
for the cases where the classification of finite dimensional irreducible
$U(\g,e)$-modules is known.

For the cases that we consider $e$ is of standard Levi type, so by a
result of Kostant in \cite[Section 2]{Ko}, we have that $U(\g_0,e)$ is
isomorphic to $S(\t)^{W_0}$, where $\t$ is a maximal toral
subalgebra of $\g_0$ and $W_0$ is the Weyl group of $\g_0$ with
respect to $\t$.  This isomorphism leads to a nice description of
finite dimensional irreducible modules for $U(\g_0,e)$ in terms of
{\em tables} associated to $e$, as explained in Sections \ref{S:A}
and \ref{S:evenchange}.
Our main results are Theorems \ref{T:hwA} and
\ref{T:changehwevenmult}, which give a combinatorial explanation of
how to pass from a table parameterizing a finite dimensional
irreducible $U(\g,e)$-module corresponding to a choice of parabolic
subalgebra $\q$ to one corresponding to a different choice of
parabolic subalgebra $\q'$. This combinatorics is given by the row
swapping operations on tables defined in \cite[Section 4]{BroG}.

The proofs of our main results depend crucially on the relationship
between finite dimensional $U(\g,e)$ modules and primitive ideals of
$U(\g)$ with associated variety $\bar{G \cdot e}$ proved by
Losev in \cite{Lo1} and \cite{Lo2}; this is discussed in
\S\ref{ss:losev}. A connection between modules for $U(\g,e)$ and
certain Whittaker modules for $U(\g)$ predicted in \cite[Conjecture
5.3]{BGK} and verified by \cite[Theorem 4.1]{Lo3} and
\cite[Proposition 3.12]{BroG} is also of importance; this is
explained in \S\ref{ss:hwthy}.   We make vital use of a theorem of
Joseph, \cite[Th\'eor\`eme 1]{Jo1}, which, in the case $\g$ is of
type $A$, tells us when two weights of $\t$ correspond to the same
primitive ideal. Other important ingredients are the notion of
``Levi subalgebras'' of $U(\g,e)$ established in \cite[Section
3]{BroG}, and the description of the component group action for the
case of rectangular nilpotent elements from \cite[Theorem
1.3]{Bro2}.

We now give a brief outline of the structure of this paper. In
Section \ref{S:recap}, we give a recollection of the theory of
finite $W$-algebras that we require later in the paper.  We prove
two general results about changing highest weight theories in
Section \ref{S:genchange}.  The main content of the paper is
Sections \ref{S:A} and \ref{S:evenchange}, in which we prove
Theorems \ref{T:hwA} and \ref{T:changehwevenmult}.  In both of these
sections we explain how {\em tables} are used to describe the
highest weight theory and the combinatorics of tables required for
changing between different highest weight theories.

\section{Review of finite $W$-algebras} \label{S:recap}

Throughout this paper we work over the field of complex numbers
$\C$; though all of our results remain valid over any algebraically
closed field of characteristic $0$. As a convention throughout this
paper, by a ``module'' we mean a finitely generated left module.

In this paper, we often consider twisted modules.  Let $A$ be an
algebra and $G$ a group that acts on $A$.  Given an $A$-module $M$
and $g \in G$, the twisted module $g \cdot M$ is equal to $M$ as a
vector space with action defined by ``$am = (g^{-1} \cdot a)m$'' for
$a \in A$ and $m \in M$.

\subsection{Definition of $U(\g,e)$} \label{ss:Wdef}
Below we recall the definition of $U(\g,e)$ via nonlinear Lie
algebras; we refer the reader to \cite[\S2.2]{BGK} for more details.

Let $G$ be a connected reductive algebraic group over $\C$; also let
$\tilde G$ be a possibly disconnected algebraic group with identity
component equal to $G$. Let $\g$ be the Lie algebra of $G$, and let
$e \in \g$ be a nilpotent element of $\g$.  Let $(\cdot|\cdot)$ be a
non-degenerate symmetric invariant bilinear form on $\g$, and define
$\chi \in \g^*$ by $\chi(x) = (e|x)$.

Given a subgroup $A$ of $G$ with Lie algebra $\a$ and $x \in \g$, we
write $A^x$ for the centralizer of $x$ in $A$ and $\a^x$ for the
centralizer of $x$ in $\a$. For $g \in G$ and $x \in \g$, we write
$g \cdot x$ for the image of $x$ under the adjoint action of $g$.

Fix an $\sl_2$-triple $(e,h,f)$ in $\g$.  We choose a maximal toral
subalgebra $\t$ of $\g$ such that $h \in \t$ and $\t^e$ is a maximal
toral subalgebra of $\g^e$. We write $\lan \cdot , \cdot \ran : \t^*
\times \t \to \C$ for the pairing between $\t^*$ and $\t$.  Let
$\Phi \sub \t^*$ be the root system of $\g$ with respect to $\t$.
Given $\alpha \in \Phi$, we write $\alpha^\vee \in \t$ for the
corresponding coroot.
Let $W$ be the Weyl group of $\g$ with
respect to $\t$.

Let
\[
  \g = \bigoplus_{i \in \Z} \g(i),
\]
be a {\em good grading} for grading for $e$ compatible with $\t$, i.e. $e \in \g(2)$,
$\g^e \sub \bigoplus_{j \geq 0} \g(j)$ and $\t \sub \g(0)$.
Good gradings for $e$ are classified in \cite{EK}; see also
\cite{BruG}. The standard example of a good grading is the {\em
Dynkin grading}, which given by $\g(i) = \{ x \in \g \mid [h,x] = i
x \}$.  The good grading is given by the $\ad h'$-eigenspace
decomposition for some $h' \in \g$; this follows from the fact that
all derivations of the derived subalgebra of $\g$ are inner.  By
\cite[Lemma 19]{BruG}, we have $h' - h \in \t^e$.

We define the following subspaces of $\g$
$$
\p = \bigoplus_{j \geq 0} \g(j), \qquad\n = \bigoplus_{j < 0} \g(j),
\qquad\h = \g(0), \qquad\k =
\g(-1).
$$
In particular, $\p$ is a parabolic subalgebra of $\g$ with Levi
factor $\h$ and $\n$ is the nilradical of the opposite parabolic.

We define a symplectic form $\lan \cdot|\cdot \ran$ on $\k$ by $\lan
x|y \ran = \chi([y,x])$. Let $ \k^{\ne} = \{x^{\ne}\mid x \in \k\}$
be a ``neutral'' copy of $\k$. We write $x^{\ne} = x(-1)^\ne$ for
any element $x \in \g$.  Now make $\k^{\ne}$ into a non-linear Lie
algebra with non-linear Lie bracket defined by $[x^\ne,y^\ne] = \lan
x|y \ran$ for $x, y \in \k$. Note that $U(\k^\ne)$ is isomorphic to
the Weyl algebra associated to $\k$ and the form $\lan \cdot|\cdot
\ran$.

We view $ \tilde{\g} = \g \oplus \k^{\ne} $ as a non-linear Lie
algebra with bracket obtained by extending the brackets already
defined on $\g$ and $\k^\ne$ to all of $\tilde{\g}$, and declaring
$[x, y^{\ne}] = 0$ for $x \in \g, y \in \k$. Then $U(\tilde{\g})\iso
U(\g) \otimes U(\k^{\ne})$. Also let $\tilde{\p} = \p \oplus
\k^{\ne}$; this is a subalgebra of $\tilde \g$ whose universal
enveloping algebra is identified with $U(\p) \otimes U(\k^{\ne})$.

We define $\tilde \n_\chi = \{x - x^\ne - \chi(x) \mid x \in \n\}$.
By the PBW theorem for $U(\tilde \g)$ we have a direct sum
decomposition $U(\tilde \g) = U(\tilde \p) \oplus U(\tilde \g)\tilde
\n_\chi$. We write $\Pr : U(\tilde \g) \to U(\tilde \p)$ for the
projection along this direct sum decomposition. We define the {\em
finite $W$-algebra}
$$
U(\g,e) = U(\tilde \p)^\n = \{u \in U(\tilde \p) \mid
\Pr([x-x^{\ne},u]) = 0 \text{ for all } x \in \n\}.
$$
It is a subalgebra of $U(\tilde \p)$ by \cite[Theorem 2.4]{BGK}.

We write $\Irr_0 U(\g,e)$ for the set of isomorphism classes of
finite dimensional irreducible $U(\g,e)$-modules.  For a finite
dimensional irreducible $U(\g,e)$-module $L$, we denote its
isomorphism class by $[L] \in \Irr_0U(\g,e)$.

\subsection{Central characters} \label{ss:centchar}

Let $Z(\g)$ denote the center of $U(\g)$ and $Z(\g,e)$ denote the
center of $U(\g,e)$. It is easy to see that the restriction of the
linear map $\Pr : U(\tilde \g) \to U(\tilde \p)$ defines an
injective algebra homomorphism $\Pr: Z(\g) \into Z(\g,e)$. As
explained in the footnote to \cite[Question 5.1]{Pr2}, this map is
also surjective, so it is an algebra isomorphism
$$
\Pr:Z(\g) \isoto Z(\g,e).
$$
We view $Z(\g)$ as a subalgebra of $U(\g,e)$-module via $\Pr$. Given
a $U(\g,e)$-module $V$, we say $V$ is of {\em central character}
$\psi:Z(\g) \to \C$ if $zv = \psi(z) v$ for all $z \in Z(\g)$ and $v
\in V$.

\subsection{The component group action} \label{ss:compgrp}

We write $H = G(0)$ for the Levi subgroup of $G$
with Lie algebra $\h = \g(0)$, so $H = G^{h'}$ (recall that the
good grading of $\g$ is the $\ad h'$-eigenspace decomposition).
The argument in the proof of \cite[Proposition 5.9]{Ja} shows that
the component group of the centralizer of $e$ in $G$, denoted by
$C(e) = G^e/(G^e)^\circ$, is naturally isomorphic to
$H^e/(H^e)^\circ$. From now on we identify $C(e) = H^e/(H^e)^\circ$.

One can check that the adjoint action of $H^e$ on $\g$ gives rise to
a well-defined action of $H^e$ on $U(\g,e)$.  It was proved by
Premet in \cite[\S2.5]{Pr2} that there is an embedding
\begin{equation} \label{e:theta}
\theta : \h^e \into U(\g,e);
\end{equation}
see also \cite[Theorem 3.3]{BGK}. Moreover, the adjoint action of
$\h^e$ on $U(\g,e)$ through this embedding coincides with
differential of the action of $H^e$ on $U(\g,e)$.

We write $\Prim_0U(\g,e)$ for the set of primitive ideals of
$U(\g,e)$ of finite codimension. The set $\Prim_0 U(\g,e)$
identifies naturally with $\Irr_0U(\g,e)$. The action of $H^e$ on
$U(\g,e)$ induces an action on $\Prim_0U(\g,e)$. Since the action of
$\h^e$ of $U(\g,e)$ coincides with the differential of the action of
$H^e$, we see that the action of $H^e$ on $\Prim_0 U(\g,e)$ factors
through $C(e)$. So we obtain an action of $C(e)$ on $\Prim_0
U(\g,e)$, and thus on $\Irr_0U(\g,e)$.

Next we note this action can also be described in terms of twisting
the action of $U(\g,e)$ on its finite dimensional irreducible
modules by elements of $C(e)$.  Let $c \in C(e)$ and $\dot c \in
H^e$ be a lift of $c$, and let $L$ be finite dimensional irreducible
$U(\g,e)$-module.    Up to isomorphism $\dot c \cdot L$  only
depends on $c$, and we define
\begin{equation} \label{e:compgp}
c \cdot [L] = [\dot c \cdot L].
\end{equation}
It is straightforward
to see that the actions of $C(e)$ on isomorphism classes of finite
dimensional irreducible $U(\g,e)$-modules via twisting
and via
the action of $C(e)$ on primitive ideals are the same.

Now let $\tilde H = \tilde G^{h'}$ be the centralizer of $h'$ in
$\tilde G$, and let $\tilde H^e$ be the centralizer of $e$ in $\tilde H$.
Then $\tilde H^e$ also acts on $U(\g,e)$.  The content of the
previous two paragraphs remains valid if we replace $C(e)$ with $\tilde C(e) = \tilde
H^e/(\tilde H^e)^\circ$.

\subsection{Skryabin's equivalence}
\label{ss:skry}

Skryabin's equivalence relates
the category $U(\g,e)\-mod$ of finitely generated $U(\g,e)$-modules
to a certain category of generalized Whittaker modules for $U(\g)$.
To state this equivalence, we require the Whittaker model definition
of $U(\g,e)$, which is outlined below.

Let $\l$ be a Lagrangian subspace of $\k$ with respect to the
symplectic form $\lan \cdot | \cdot \ran$. Then define $\m =
\bigoplus_{j \le -2} \g(j) \oplus \l$ and $\m_\chi = \{x - \chi(x)
\mid x \in \m\} \sub U(\g)$.  Then $Q_\chi \iso U(\g)/U(\g)\m_\chi$
is a left $U(\g)$-module.  There is a natural isomorphism
$\End_{U(\g)}(Q_\chi)^{\op} \isoto U(\g,e)$, by \cite[Theorem
2.4]{BGK}. The algebra $\End_{U(\g)}(Q_\chi)^{\op}$ is the Whittaker
model definition of the finite $W$-algebra associated to $\g$ and
$e$.

Now $Q_\chi$ is naturally a right
$\End_{U(\g)}(Q_\chi)^{\op}$-module and thus can be viewed as a
right $U(\g,e)$-module. Therefore, we can define the $U(\g)$-module
$\cS(M) = Q_\chi \otimes_{U(\g,e)} M$ for $M \in U(\g,e)\-mod$. Let
$\Wh(\g,\m_\chi)$ be the category of finitely generated
$U(\g)$-modules on which $\m_\chi$ acts locally nilpotently.  For $M
\in U(\g,e)\-mod$ it is easy to check that $\cS(M) \in
\Wh(\g,\m_\chi)$. Skryabin's equivalence from \cite{Sk} says that
the functor
$$
\cS : U(\g,e)\-mod \to \Wh(\g,\m_\chi)
$$
is an equivalence of categories.  A quasi-inverse is given by the
functor
$$
N \mapsto N^{\m_\chi} = \{n \in N \mid xn = \chi(x)n \text{ for all
} x \in \m\}
$$
for $N \in \Wh(\g,\m_\chi)$.

\subsection{Losev's map between ideals} \label{ss:losev}

In \cite{Lo1} Losev constructs a map $\cdot ^\dagger$ from the set
of ideals of $U(\g,e)$ to the set of ideals of $U(\g)$.  By
\cite[Theorem 1.2.2]{Lo1}, this map restricts to a
surjection
\begin{equation} \label{e:dagger}
I \mapsto I^\dagger : \Prim_0 U(\g,e) \twoheadrightarrow \Prim_e
U(\g),
\end{equation}
where
$\Prim_{e}U(\g)$ denotes the primitive ideals of $U(\g)$ with
associated variety equal to $\bar{G \cdot e}$.  For a definition of
associated varieties, see for example \cite[Section 9]{Ja}.

Furthermore, in \cite[Theorem 1.2.2]{Lo2} Losev proves that the fibres
of the map in \eqref{e:dagger} are precisely the $C(e)$-orbits in
$\Prim_0 U(\g,e)$, for the action of $C(e)$ explained in
\S\ref{ss:compgrp}.

By \cite[Theorem 1.2.2]{Lo1}, the map $\cdot ^\dagger$ restricted to
$\Prim_0 U(\g,e)$ can be described as follows.  Let $I \in \Prim_0
U(\g,e)$ and let $L$ be a finite dimensional irreducible
$U(\g,e)$-module with $\Ann_{U(\g,e)}(L) = I$.  Then
$$
I^\dagger = \Ann_{U(\g)}(\cS(L)).
$$

In \cite[Theorem 1.2.2]{Lo1}, it is proved that if $L$ is an
irreducible $U(\g,e)$-module with central character $\psi : Z(\g)
\to \C$, then $\Ann_{U(\g,e)}(L)^\dagger \cap Z(\g) = \ker \psi$,
where $Z(\g)$ is viewed as a subalgebra of $U(\g,e)$ as in
\S\ref{ss:centchar}. Thus $\cdot^\dagger$ preserves central
characters.

\subsection{Review of highest weight theory} \label{ss:hwthy}

Highest weight theory for finite $W$-algebras
was introduced in
\cite[Section 4]{BGK}.  In this paper we restrict to the case where
$e$ is of standard Levi type, as defined below.

We let $\g_0 = \{x \in \g \mid [t,x] = 0 \text{ for all } t \in
\t^e\}$ be the centralizer of $\t^e$ in $\g$.  Then $\g_0$ is a Levi
subalgebra of $\g$ and $e$ is a distinguished nilpotent element of
$\g_0$.  We restrict to the case that $e$ is of {\em standard Levi
type}, which means that $e$ is regular nilpotent in $\g_0$.  We
write $\Phi_0$ for the root system of $\g_0$ with respect to $\t$.

We can form the $\t^e$-weight space decomposition
$$
\g = \g_0 \oplus \bigoplus_{\alpha \in \Phi^e} \g_\alpha
$$
of $\g$, where $\Phi^e \sub (\t^e)^*$ and $\g_\alpha = \{x \in \g \mid
[t,x] = \alpha(t)x \text{ for all } t \in \t^e\}$.  Then $\Phi^e$ is
a restricted root system; see \cite[Sections 2 and 3]{BruG} for information
on restricted root systems.

We choose a parabolic subalgebra $\q$ of $\g$ with Levi factor
$\g_0$. The parabolic subalgebra $\q$ gives a system $\Phi^e_+$ of
positive roots in $\Phi^e$, namely, $\Phi^e_+ = \{\alpha \in \Phi^e
\mid \g_\alpha \sub \q\}$.  The highest weight theory explained
below depends on this choice of $\q$, and this dependency is the
main topic of study in this article.

Note that $\t^e \sub \h^e$ embeds in $U(\g,e)$ via the map $\theta$
from \eqref{e:theta}. Therefore, we have a $\t^e$-weight space
decomposition
\[
U(\g,e) = U(\g,e)_0 \oplus \bigoplus_{\alpha
\in \Z \Phi^e \setminus \{0\} } U(\g,e)_\alpha.
\]

The zero weight space $U(\g,e)_0$ is a subalgebra of $U(\g,e)$ and
we define $U(\g,e)_\sharp$ to be the left ideal of $U(\g,e)$
generated by $U(\g,e)_\alpha$ for $\alpha \in \Phi^e_+$. Then
$U(\g,e)_{0,\sharp} = U(\g,e)_0 \cap U(\g,e)_{\sharp}$ is a two
sided ideal of $U(\g,e)_0$ so we can form the quotient
$U(\g,e)_0/U(\g,e)_{0,\sharp}$.

By \cite[Theorem 4.3]{BGK}, there is an isomorphism
\begin{equation} \label{e:hwiso}
U(\g,e)_0/U(\g,e)_{0,\sharp} \iso U(\g_0,e).
\end{equation}
This isomorphism is central to the development of the highest weight
theory  since it is used to define Verma modules, as we explain below.

Since $e$ is regular in $\g_0$, we have that $\p_0 = \p \cap \g_0$
is a Borel subalgebra of $\g_0$; we write $\b_0 = \p_0$ and
$\Phi_0^+ \sub \Phi_0$ for the system of positive roots
corresponding to $\b_0$. Then we set $\b_\q = \b_0 \oplus \q_u$,
where $\q_u$ denotes the nilradical of our parabolic $\q$, so that
$\b_\q$ is a Borel subalgebra of $\g$.
We also need another Borel subalgebra, $\tilde \b_\q = \tilde \b_0
\oplus \q_u$ where $\tilde \b_0$ is the opposite Borel to $\b_0$ in
$\g_0$. We let $\rho_\q$ and $\tilde \rho_\q$ denote the half sum of
the positive roots corresponding to $\b_\q$ and $\tilde \b_\q$
respectively.

Since $e$ is regular in $\g_0$, a result of Kostant in
\cite[Section 2]{Ko} gives that $U(\g_0,e) \iso S(\t)^{W_0}$, where
$W_0$ denotes the Weyl group of $\g_0$ with respect to $\t$. An
explicit isomorphism
$$
\xi_{-\tilde \rho_\q}: U(\g_0,e) \isoto S(\t)^{W_0}
$$
is given in \cite[Lemma 5.1]{BGK}, where $\xi_{-\tilde \rho_\q}$ is
the composition of the natural projection $U(\b_0) \to S(\t)$ with
the shift  $S_{-\tilde \rho_\q} : S(\t) \to S(\t)$, where
$S_{-\tilde \rho_\q}(t) = t - \tilde \rho_\q(t)$ for $t \in \t$.

The finite dimensional irreducible modules for $S(\t)^{W_0}$ are all
1-dimensional and are indexed by the set $\cL = \t^*/W_0$ of
$W_0$-orbits in $\t^*$. Given $\Lambda \in \cL$ we let $V_\Lambda$
be the $U(\g_0,e)$-module corresponding to $\Lambda$ through
$\xi_{-\tilde \rho_\q}$. We define the Verma module
\begin{equation} \label{e:verma}
M(\Lambda,\q) = (U(\g,e)/U(\g,e)_\sharp)
\otimes_{U(\g_0,e)} V_\Lambda,
\end{equation}
where $U(\g,e)/U(\g,e)_\sharp$ is viewed as a right
$U(\g_0,e)$-module via the isomorphism from \eqref{e:hwiso}.

By \cite[Theorem 4.5]{BGK}, $M(\Lambda,\q)$ has a unique maximal
submodule and we write $L(\Lambda,\q)$ for the irreducible quotient.
Moreover, any finite dimensional irreducible $U(\g,e)$-module is
isomorphic to $L(\Lambda,\q)$ for some $\Lambda \in \cL$, and
$L(\Lambda,\q) \iso L(\Lambda',\q)$ if and only if $\Lambda =
\Lambda'$.

Let $\Psi : Z(\g) \to S(\t)^W$ be the Harish-Chandra isomorphism
defined by
$$
z \equiv S_{\rho_\q}(\Psi(z)) \mod U(\g)\b_{\q,u},
$$
where $\b_{\q,u}$ denotes the nilradical of $\b_\q$.  Under this
isomorphism, the central character of $L(\Lambda,\q)$ corresponds to
the $W$-orbit in $\t^*$ containing $\Lambda$ by \cite[Corollary
4.8]{BGK}.

We let $\mathcal{L}^+_\q = \{\Lambda \in \mathcal{L} \mid L(V,\q)
\text{ is finite dimensional}\}$.  So this set parameterizes the
isomorphism classes of finite dimensional irreducible
$U(\g,e)$-modules.  For a different choice of parabolic subalgebra
$\q'$ of $\g$ with Levi factor $\g_0$, we obtain another subset
$\mathcal{L}^+_{\q'}$ of $\cL$ that parameterizes the isomorphism
classes of finite dimensional irreducible $U(\g,e)$-modules.  Thus
there is a bijection $f : \cL^+_\q \isoto \cL^+_{\q'}$ such that
$L(\Lambda,\q) \iso L(f(\Lambda),\q')$.  The main theorems of this
paper are Theorems \ref{T:hwA} and \ref{T:changehwevenmult}, which
give a combinatorial description of this bijection in certain cases.

Given a $U(\g,e)$-module $V$ we say that $v \in V$ is a highest
weight vector for (the parabolic subalgebra) $\q$ if $uv = 0$ for
all $u \in U(\g,e)_\sharp$, and $v$ is an eigenvector for every
element of $U(\g,e)_0$.  In this case
$\lan v \ran$ has the structure of a $U(\g_0,e)$-module, which is
isomorphic to $V_\Lambda$ for some $\Lambda \in \cL$, and we say
that $v$ is of highest weight $\Lambda$. Since $L(\Lambda,\q)$ is
irreducible, it has a unique, up to scalar multiplication, highest
weight vector for $\q$ (of highest weight $\Lambda$). Given another
parabolic subalgebra $\q'$ with Levi factor $\g_0$, we can define
highest weight vectors for $\q'$ analogously.

As explained in \S\ref{ss:compgrp}, there is an action of $\tilde
C(e)$ on the set of isomorphism classes of finite dimensional
irreducible $U(\g,e)$-modules given by \eqref{e:compgp}.  This gives
an action of $\tilde C(e)$ on $\cL^+_\q$ defined by
$$
c \cdot [L(\Lambda,\q)] = [L(c \cdot \Lambda,\q)],
$$
for $c \in \tilde C(e)$ and $\Lambda \in \cL^+_\q$. To be clear,
here we are defining an action of $\tilde C(e)$ on a subset of
$\t^*/W_0$. In some cases it is possible to define a more natural
action of $\tilde C(e)$ on $\t^*/W_0$, however in general these
actions are not compatible.

Next in \eqref{e:hwdagger} we state a relationship between the
highest weight theory and the map $\cdot^\dagger$ from
\S\ref{ss:losev}. This is due to an equivalence of categories
between an analogue of the BGG category $\cO$ for $U(\g,e)$ and a
certain category of generalized Whittaker modules for $U(\g)$, which
was predicted in \cite[Conjecture 5.3]{BGK}.   This conjecture was
verified by \cite[Theorem 4.1]{Lo3}, but in the setting of highest
weight theory defined in a different way.  In \cite[Proposition
3.12]{BroG} it is shown that the Verma modules defined in the
different highest weight theories coincide, thus completing the
verification of \cite[Conjecture 5.3]{BGK}.

Let $\Lambda \in \cL^+$ and take $\lambda \in \Lambda$ such that
$\lan \lambda,\alpha^\vee \ran \notin \Z_{> 0}$ for all $\alpha \in
\Phi_0^+$. Let $L(\lambda,\b_\q)$ be the irreducible highest weight
$U(\g)$-module with highest weight $\lambda - \rho_\q$ with respect
to the Borel subalgebra $\b_\q$. Then using \cite[Conjecture
5.3]{BGK}, \cite[Theorem 5.1]{MS} and \cite[Theorem 1.2.2]{Lo1}, we
obtain that
\begin{equation} \label{e:hwdagger}
\Ann_{U(\g,e)}(L(\Lambda,\q))^\dagger =
\Ann_{U(\g)}(L(\lambda,\b_\q)).
\end{equation}

\subsection{ Parabolic highest weight theories}
We end this section by briefly discussing a ``parabolic
generalization'' of the highest weight theory from \cite[Section 3]{BroG}.
To do this we first recall that a subalgebra $\s$ of $\t^e$ is called a {\em full subalgebra} if
$\s$ is equal to the centre of the Levi subalgebra $\g^\s = \{x \in \g \mid [t,x] = 0
\text{ for all } t \in \s\}$ of $\g$.

By \cite[Theorem 3.2]{BroG}, there is an isomorphism generalizing
that of \eqref{e:hwiso} between a subquotient of $U(\g,e)$ and
$U(\g^\s,e)$.  This is obtained by taking $\s$-weight spaces in
$U(\g,e)$ rather than $\t^e$ weight spaces. To define parabolic
Verma modules, we need to use the parabolic subalgebra $\q_\s$ which
has $\g^\s$ as its Levi factor and contains $\q$. Then given an
irreducible finite dimensional module $V$ for $U(\g^\s,e)$ we can
define a parabolic Verma module $M_\s(V,\q_\s)$ for $U(\g,e)$, which
has an irreducible head $L_\s(V,\q_\s)$ as in \cite[\S3.3]{BroG}.
We denote these modules by $M_\s(V,\q)$ and $L_\s(V,\q)$.

The version of \eqref{e:hwiso} in the case ``$\g = \g^\s$'' allows
us to define Verma modules for $U(\g^\s,e)$.
For $\Lambda \in \cL = \t^*/W_0$ we can define the Verma module
$M^\s(\Lambda,\q)$ for $U(\g^\s,e)$ in analogy to \eqref{e:verma},
see \cite[\S3.3]{BroG} for details.  We write $L^\s(\Lambda,\q)$ for
the irreducible head of $M^\s(\Lambda,\q)$.

The important point for us is the transitivity result
\cite[Proposition 3.6]{BroG}. This says that if $\Lambda \in
\cL^+_\q$ (so that $L(\Lambda,\q)$ is finite dimensional), then
$L^\s(\Lambda,\q)$ is finite dimensional and
\begin{equation} \label{e:translevi}
L_\s(L^\s(\Lambda,\q),\q) \iso L(\Lambda,\q).
\end{equation}

\section{Generalities about changing height weight theories} \label{S:genchange}

In this section we prove two general results about changing highest
weight theories. In \S\ref{ss:changehw} we prove Theorem
\ref{T:hwgen}, which says how to pass between highest weight
theories up to the action of $C(e)$. Then in \S\ref{ss:rest} we
prove Proposition \ref{P:restchange}, which deals with the case
where the parabolic subalgebras are conjugate under the action of
the {\em restricted Weyl group} $W^e$.

\subsection{Changing the highest weight theory up to the action of
$C(e)$} \label{ss:changehw}

Let $\q$ and $\q'$ be parabolic subalgebras of $\g$ with Levi factor
$\g_0$.
Let $\b_{\q'}$ be the Borel subalgebra of $\g$ given by $\b_{\q'} =
\b_0 \oplus \q'_u$,  where $\q'_u$ is the nilradical of $\q'$.
Define $\rho_{\q'}$ to be the half sum of the positive roots
determined by $\b_{\q'}$. Let $\Lambda \in \cL^+_\q$ and let
$\Lambda'  \in \cL^+_{\q'}$.  Take $\lambda \in \Lambda$ and
$\lambda' \in \Lambda'$ with $\lan \lambda,\alpha^\vee \ran \notin
\Z_{> 0}$ and $\lan \lambda',\alpha^\vee \ran \notin \Z_{> 0}$ for
all $\alpha \in \Phi_0^+$.  We denote the highest weight
$U(\g)$-module with highest weight $\lambda' - \rho_{\q'}$ with
respect to $\b_{\q'}$ by $L(\lambda',\b_{\q'})$. Finally let $w \in
W$ be such that $w \cdot \b_{\q'} = \b_\q$

\begin{Theorem} \label{T:hwgen}
In the notation given above we have,
$[L(\Lambda,\q)]$ and $[L(\Lambda',\q')]$ lie in the same
$C(e)$-orbit if and only if $\Ann_{U(\g)} L(\lambda,\b_\q) =
\Ann_{U(\g)} L(w \lambda',\b_\q)$.
\end{Theorem}

\begin{proof}
From \eqref{e:hwdagger} we have
$$
\Ann_{U(\g,e)}L(\Lambda,\q)^\dagger = \Ann_{U(\g)}L(\lambda,\b_\q).
$$
Similarly, we have
$$
\Ann_{U(\g,e)} L(\Lambda',\q') ^\dagger = \Ann_{U(\g)} L(\lambda',\b_{\q'}) .
$$
Also, if $\dot w \in N_G(\t)$ is a lift of $w$, then $\dot w \cdot
L(\lambda',\b_{\q'}) \iso L(w\lambda',\b_\q)$, because the highest
weight vector with respect to $\b_{\q'}$ in $L(\lambda',\b_{\q'})$
is a highest weight vector of highest weight $w \lambda'$ with
respect to $\b_\q$  in $\dot w \cdot L(\lambda',\b_{\q'})$.
Thus
\[
\Ann_{U(\g)} L(\lambda',\b_{\q'}) = \dot w ^{-1} \cdot \Ann_{U(\g)}
L(w \lambda', \b_\q) = \Ann_{U(\g)} L(w \lambda', \b_\q).
\]
So, recalling the discussion from $\S2.5$, we see that the theorem follows
from \cite[Theorem 1.2.2]{Lo2}.
\end{proof}

\subsection{Changing highest weight theory with the restricted Weyl
group} \label{ss:rest} The restricted Weyl group $W^e =
N_{G^e}(\t^e)/Z_{G^e}(\t^e)$ is defined in \cite[Section 3]{BruG},
where $N_{G^e}(\t^e)$
is the normalizer of $\t^e$ in $G^e$ and $Z_{G^e}(\t^e)$
is the
centralizer of $\t^e$ in $G^e$.

As in \S\ref{ss:compgrp} we let $H = G(0)$ be the Levi subgroup of
$G$ with Lie algebra $\h = \g(0)$.  Also we let $R$ be the unipotent
subgroup of $G$ with Lie algebra $\bigoplus_{j > 1} \g(j)$.  Then we
have a Levi decomposition $G^e = H^e \ltimes R^e$; this can be
proved using the argument in \cite[Proposition 5.9]{Ja}. We see that
$N_{G^e}(\t^e) = N_{H^e}(\t^e) \ltimes Z_{R^e}(\t^e)$, because $\t^e
\sub \h^e$. This leads to an isomorphism $W^e \iso
N_{H^e}(\t^e)/Z_{H^e}(\t^e)$.

We can view $W^e$ naturally as a subgroup of $\GL(\t^e)$. Viewing
$W$ as a subgroup of $\GL(\t)$, we note that $W_0$ centralizes
$\t^e$ and $N_W(W_0)$ normalizes $\t^e$.  Thus $N_W(W_0)$ can be
viewed as a subgroup of $\GL(\t^e)$.  Thanks to \cite[Lemma
14]{BruG}, we have $W^e = N_W(W_0)/W_0$ as subgroups of $\GL(\t^e)$.

An element of $Z_{G^e}(\t^e)$ normalizes any parabolic
subalgebra of $\g$ with Levi factor $\g_0$, and any element of
$N_{G^e}(\t^e)$ normalizes $\g_0$. Therefore, $W^e$ acts on the set
of parabolic subalgebras of $\g$ with Levi factor $\g_0$. Below we
explain how to pass between different highest weight theories
corresponding to parabolic subalgebras that are conjugate by $W^e$.

The adjoint action of $H^e$ on $U(\g,e)$, explained in
\S\ref{ss:compgrp} restricts to an action of $N_{H^e}(\t^e)$ on
$U(\g,e)$.  Thus we can twist $U(\g,e)$-modules by elements of
$N_{H^e}(\t^e)$.
The adjoint action of $N_{H^e}(\t^e)$ on $\g$ also gives rise to an
action of $N_{H^e}(\t^e)$ on $U(\g_0,e)$. Thus we can twist
$U(\g_0,e)$-modules by elements of $N_{H^e}(\t^e)$. Let $G_0$ be the
centralizer of $\t^e$ in $G$ so the Lie algebra of $G_0$ is $\g_0$;
we note that $Z_{G^e}(\t^e)$ is the centralizer of $e$ in $G_0$. Now
$e$ is regular in $\g_0$, and $Z_{H^e}(\t^e)$ is a Levi factor of
$Z_{G^e}(\t^e)$, thus $Z_{H^e}(\t^e)$ is equal to the centre of
$G_0$; this follows from standard results about the centralizers of
regular nilpotent elements.
Therefore, we see that the action of
$Z_{H^e}(\t^e)$ on $U(\g_0,e)$ is trivial and thus we can twist
$U(\g_0,e)$-modules by elements of $W^e$. Hence, we obtain an action
of $W^e$ on $\cL = \t^*/W_0$. From the proof of \cite[Lemma
14]{BruG}, we see that through the isomorphism $W^e \iso
N_W(W_0)/W_0$ this action coincides with the natural action of
$N_W(W_0)/W_0$ on $\t^*/W_0$.

For the remainder of this subsection we fix $\q$ a parabolic
subalgebra of $\g$ with Levi factor $\g_0$.  Let $\Lambda \in
\mathcal{L}^+_\q$, let $v_+$ be the highest weight vector in
$L(\Lambda,\q)$ for $\q$, and let $h \in N_{H^e}(\t^e)$. In $h \cdot
L(\Lambda,\q)$, we have that $v_+$ is a highest weight vector for
$\q' = h \cdot \q$. Therefore, if $h \in Z_{H^e}(\t^e)$, then $v_+$
is a highest weight vector for $\q$. Since the action of
$Z_{H^e}(\t^e)$ on $U(\g_0,e)$ is trivial
we thus see that $h \cdot L(\Lambda,\q) \iso L(\Lambda,\q)$. Hence,
we obtain an action of $W^e$ on $\Irr_0 U(\g,e)$.

The following lemma is immediate from the discussion above.

\begin{Lemma}
There are actions of $W^e$ on $\cL$ and $\Irr_0U(\g,e)$. For $w
\in W^e$ and $\Lambda \in \cL_\q^+$, we have
$$
w \cdot [L(\Lambda,\q)] = [L(w \cdot \Lambda, w \cdot \q)].
$$
\end{Lemma}

Let $\k = \g^h$.  We note that $\k^e$ is reductive, see
\cite[Proposition 5.9]{Ja}, and that $\k^e \sub \h^e$, because $h'-h
\in \t^e$.  Also $\t^e$ is a maximal toral subalgebra of $\k^e$. We
decompose $\k^e$ in to $\t^e$-weight spaces
$$
\k^e = \k^e_0 \oplus \bigoplus_{\alpha \in (\Phi^e)^\circ}
\k^e_\alpha,
$$
where $(\Phi^e)^\circ \sub \Phi^e$ is the root system of $\k^e$ with
respect to $\t^e$.  Then $(\Phi^e)^\circ_+ = (\Phi^e)^\circ \cap
\Phi^e_+$ is a system of positive roots of positive roots in
$(\Phi^e)^\circ$.  As in \cite[Section 3]{BruG} we define $Z^e$ to be the
stabilizer in $W^e$ of the dominant chamber in $\R\Phi_e$ determined
by $(\Phi^e)^\circ_+$.

Let $(W^e)^\circ = N_{(H^e)^\circ}(\t^e)/Z_{(H^e)^\circ}(\t^e)$. By
\cite[Lemma 15]{BruG}, we have that $W^e \iso Z^e \ltimes
(W^e)^\circ$. Moreover, the inclusion $N_{H^e}(\t^e) \into H^e$
induces an isomorphism
$$
\iota : Z^e \isoto H^e/(H^e)^\circ Z_{H^e}(\t^e).
$$
Also
$H^e/(H^e)^\circ Z_{H^e}(\t^e)$ is a quotient of the component group
$C(e)$ via the natural map
$$
\kappa : C(e) \onto H^e/(H^e)^\circ Z_{H^e}(\t^e).
$$
 Let $z \in
Z^e$ and let $c \in C(e)$ such that $\iota(z) = \kappa(c)$.  Then by
the definitions of the actions, we see have $[L(c \cdot \Lambda,\q)]
= c \cdot [L(\Lambda,\q)] = z \cdot [L(\Lambda,\q)]$, for $\Lambda
\in \cL_\q^+$.
Since $N_{(H^e)^\circ}(\t^e) \sub (H^e)^\circ$, we have $[w \cdot
L(\Lambda,\q)] = [L(\Lambda,\q)]$ for any $w \in (W^e)^\circ$ and
$\Lambda \in \cL_\q^+$.

Putting together the discussion above we arrive at the following
proposition.

\begin{Proposition} \label{P:restchange}
Let $\q,\q'$ be parabolic subalgebras of $\g$ with Levi factor
$\g_0$ and with $w \cdot \q' = \q$ for some $w \in W^e$, and let
$\Lambda \in \cL^+_\q$. Write $w = zv \in W^e$, where $z \in Z^e$
and $v \in (W^e)^\circ$, and let $c \in C(e)$ such that $\iota(z) =
\kappa(c)$. Then
$$
[L(\Lambda,\q)] = [L(c \cdot \Lambda, \q')].
$$
\end{Proposition}

In \S\ref{ss:change} we require the restricted Weyl group $\tilde
W^e = N_{\tilde G^e}(\t^e)/Z_{\tilde G^e}(\t^e)$ for $\tilde G$.  It
is easy to check that everything above holds with $\tilde W^e$ in
place of $W^e$.

\section{Changing highest weight theories in type A} \label{S:A}

The goal of this section is to prove Theorem \ref{T:hwA}, which
explains how to construct the bijection between parameterizing sets
of finite dimensional irreducible $U(\g,e)$-modules for different
highest weight theories when $\g$ is of type $A$. First we recall
the classification of finite dimensional irreducible
$U(\g,e)$-modules in \S\ref{ss:classA}. Next, in \S\ref{ss:frames},
we recall some definitions from \cite[Section 4]{BroG} regarding
frames and tables, which give the combinatorics for the description
of the highest weight theories. Finally, in \S\ref{ss:changeA}, we
state and prove Theorem \ref{T:hwA}.

\subsection{The classification of finite dimensional irreducible
$U(\g,e)$-modules} \label{ss:classA}

We let $\g = \gl_n$ and let $\{e_{i,j} \mid 1 \leq i,j \leq n\}$ be
the standard basis of $\g$.  Write $(\cdot | \cdot)$ for the trace
form on $\g$. Let $\t$ be the maximal toral subalgebra of diagonal
matrices.  Define $\eps_i \in \t^*$ to be dual to $e_{i,i}$.  The
Weyl group $W$ of $\g$ with respect to $\t$ is the symmetric group
$S_n$.

We recall that nilpotent $G$-orbits are parameterized by partitions
of $n$.  Also we recall that the centralizer in $G$ of any nilpotent
element $e \in \g$ is connected, so that $C(e)$ is trivial.

To define $U(\g,e)$ we require a good grading for $e$. Good gradings
for $e$ were classified in \cite{EK} using {\em pyramids}. A pyramid
is a finite collection of boxes in the plane such that:
\begin{itemize}
    \item[-] the boxes are arranged in connected rows;
    \item[-] each box is 2 units by 2 units;
    \item[-] each box is centred at a point in $\Z^2$;
    \item[-] if a box centred at $(i,j)$ is not in the bottom row then
        there is a box in the pyramid centered at $(i, j-2)$ or
        there are two boxes in the pyramid centered at
        $(i-1,j-2)$ and $(i+1,j-2)$.
\end{itemize}
For example
\begin{equation} \label{EQ:pyr1}
    \begin{array}{c}
\begin{picture}(80,60)
\put(30,60){\line(1,0){20}}
\put(0,0){\line(1,0){80}}
\put(0,20){\line(1,0){80}}
\put(10,40){\line(1,0){40}}
\put(0,0){\line(0,1){20}}
\put(20,0){\line(0,1){20}}
\put(40,0){\line(0,1){20}}
\put(60,0){\line(0,1){20}}
\put(80,0){\line(0,1){20}}
\put(10,20){\line(0,1){20}}
\put(30,20){\line(0,1){20}}
\put(50,20){\line(0,1){20}}
\put(30,40){\line(0,1){20}}
\put(50,40){\line(0,1){20}}
\end{picture}
\end{array}
\end{equation}
is a pyramid.

Let $\bp = (p_1 \ge p_2 \ge \dots \ge p_m)$ be a partition of $n$
and let $P$ be a pyramid with row lengths given by $\bp$.

The {\em coordinate table of $P$} is obtained by filling the boxes
in $P$ with entries $1, \dots, n$ filled in from top to bottom and
from left to right and is denoted $K$. For example if $P$ is the
pyramid in \eqref{EQ:pyr1} then the coordinate table of $P$ is
\begin{equation} \label{e:coord}
    K =
    \begin{array}{c}
\begin{picture}(80,60)
\put(30,60){\line(1,0){20}}
\put(0,0){\line(1,0){80}}
\put(0,20){\line(1,0){80}}
\put(10,40){\line(1,0){40}}
\put(0,0){\line(0,1){20}}
\put(20,0){\line(0,1){20}}
\put(40,0){\line(0,1){20}}
\put(60,0){\line(0,1){20}}
\put(80,0){\line(0,1){20}}
\put(10,20){\line(0,1){20}}
\put(30,20){\line(0,1){20}}
\put(50,20){\line(0,1){20}}
\put(30,40){\line(0,1){20}}
\put(50,40){\line(0,1){20}}
\put(40,50){\makebox(0,0){{1}}}
\put(20,30){\makebox(0,0){{2}}} \put(40,30){\makebox(0,0){{3}}}
\put(10,10){\makebox(0,0){{4}}} \put(30,10){\makebox(0,0){{5}}}
\put(50,10){\makebox(0,0){{6}}} \put(70,10){\makebox(0,0){{7}}}
\end{picture}
\end{array}.
\end{equation}
Define $e = \sum e_{i,j} \in \g$ where we sum over all $i,j$ such
that $j$ is the right neighbour of $i$ in $K$, so $e$ is a nilpotent
element of $\g$ with Jordan type $\bp$.  In the example above we
have
$$
e = e_{2,3} + e_{4,5} + e_{5,6} + e_{6,7}.
$$
For $i = 1,\dots,n$, we write $\col(i)$ for the $x$-coordinate of
the center of the box in $K$ containing $i$.
Let
$$
\g(k) = \lan e_{i,j} \mid \col(j) - \col(i) = k \ran.
$$
Then $\g = \bigoplus_{k \in \Z} \g(k)$ is a good grading for $e$ and
all good gradings for $e$ occur in this way; we refer to
\cite[Section 4]{EK} and \cite[Section 6]{BruG} for more information on good
gradings for $\gl_n$. Now the finite $W$-algebra $U(\g,e)$ can be
defined as in \S\ref{ss:Wdef}.

For $i = 1,\dots,n$ we write $\row(i)$ for the row of $K$ in which
$i$ appears where we label the rows of $K$ with $1,\dots,m$ from top
to bottom. Then we have
$$
\g_0 = \lan e_{i,j} \mid \row(i) = \row(j) \ran,
$$
and
$$
\b_0 = \lan e_{i,j} \mid \row(i) = \row(j) \text{ and } \col(i) \le
\col(j) \ran.
$$
We take
$$
\q = \lan e_{i,j} \mid \row(i) \le \row(j) \ran,
$$
as our choice of parabolic subalgebra of $\g$ with Levi subalgebra
$\g_0$.

For the rest of this paper we use the partial order on $\C$ where $a
\leq b$ if $b-a \in \Z_{\ge 0}$. We say that $P$ is {\em justified}
if the boxes are aligned in columns.
We let $\Tab(P)$ denote the set of fillings of $P$ with complex
numbers.
We define the {\em left justification} of $P$ to be the diagram
$l(P)$ obtained from $P$ by left justifying the rows; given $A \in
\Tab(P)$, we define $l(A) \in \Tab(l(F))$ similarly. For example, if
$$
    A =
    \begin{array}{c}
\begin{picture}(80,60)
\put(30,60){\line(1,0){20}}
\put(0,0){\line(1,0){80}}
\put(0,20){\line(1,0){80}}
\put(10,40){\line(1,0){40}}
\put(0,0){\line(0,1){20}}
\put(20,0){\line(0,1){20}}
\put(40,0){\line(0,1){20}}
\put(60,0){\line(0,1){20}}
\put(80,0){\line(0,1){20}}
\put(10,20){\line(0,1){20}}
\put(30,20){\line(0,1){20}}
\put(50,20){\line(0,1){20}}
\put(30,40){\line(0,1){20}}
\put(50,40){\line(0,1){20}}
\put(40,50){\makebox(0,0){{5}}}
\put(20,30){\makebox(0,0){{-1}}} \put(40,30){\makebox(0,0){{3}}}
\put(10,10){\makebox(0,0){{-3}}} \put(30,10){\makebox(0,0){{1}}}
\put(50,10){\makebox(0,0){{1}}} \put(70,10){\makebox(0,0){{4}}}
\end{picture}
\end{array},
$$
then
\[
    l(A) =
    \begin{array}{c}
\begin{picture}(80,60)
\put(0,60){\line(1,0){20}}
\put(0,0){\line(1,0){80}}
\put(0,20){\line(1,0){80}}
\put(0,40){\line(1,0){40}}
\put(0,0){\line(0,1){20}}
\put(20,0){\line(0,1){20}}
\put(40,0){\line(0,1){20}}
\put(60,0){\line(0,1){20}}
\put(80,0){\line(0,1){20}}
\put(0,20){\line(0,1){20}}
\put(20,20){\line(0,1){20}}
\put(40,20){\line(0,1){20}}
\put(20,40){\line(0,1){20}}
\put(0,40){\line(0,1){20}}
\put(10,50){\makebox(0,0){{5}}}
\put(10,30){\makebox(0,0){{-1}}} \put(30,30){\makebox(0,0){{3}}}
\put(10,10){\makebox(0,0){{-3}}} \put(30,10){\makebox(0,0){{1}}}
\put(50,10){\makebox(0,0){{1}}} \put(70,10){\makebox(0,0){{4}}}
\end{picture}
\end{array}.
\]

The {\em row equivalence class} of $A \in \Tab(P)$ is obtained by
taking all possible permutations of the entries in the rows of $A$;
we write $\bar A$ for the row equivalence class of $A$. We write
$\Row(P)$ for the set of row equivalence classes of elements in
$\Tab(P)$.  For justified $P$, we say $A \in \Tab(P)$ is {\em column
strict} if the entries are strictly decreasing down columns with
respect to the partial order defined above.

To each $A \in \Tab(P)$ we associate a weight $\lambda_A = \sum a_i
\eps_i \in \t^*$, where $a_i$ is the number in the box of $A$ which
occupies the same position as $i$ in $K$.  For example, with $K$ and $A$ as
above
we have
$$
\lambda_A = 5\eps_1 - \eps_2 +3\eps_3 -3 \eps_4 + \eps_5 +  \eps_6 +
4\eps_7 .
$$
Let $\Lambda_A$ be the $W_0$-orbit of $\lambda_A$.  We note that
$W_0$ is isomorphic to $S_{p_1} \times \dots S_{p_m}$ and the action
of $W_0$ on $\t^*$ corresponds to $W_0$ acting on tables by
permuting entries in rows. Thus $\Lambda_A$ corresponds to the row
equivalence class $\bar A$ of $A$. We write $L(\bar A)$ for the highest
weight irreducible $U(\g,e)$-module $L(\Lambda_A,\q)$, as defined in
\S\ref{ss:hwthy}.

Now we are ready to state the classification of finite dimensional
irreducible $U(\g,e)$-modules, as discovered by Brundan and
Kleshchev in \cite{BK}.

\begin{Theorem}[{\cite[Theorem 7.9]{BK}}] \label{T:classfdA}
    \[
    \left\{ L(\bar A) \mid \bar A \in \Row(P),
    \bar {l(A)} \text{ contains an element which is
    column strict }
    \right\}
    \]
    is a complete set of pairwise distinct isomorphism classes
    of finite dimensional simple $U(\g,e)$-modules.
\end{Theorem}

\subsection{Frames and tables} \label{ss:frames}

We recall some definitions about frames and tables, for more details
see \cite[Section 4]{BroG}.  We note that
we use different notation for row swapping here.

A {\em box diagram} is a finite connected collection of boxes
arranged in rows in the plane. Note that the symmetric group $S_m$
acts naturally on the set of box diagrams with $m$ rows by permuting
rows.  We number the rows in a box diagram from top to bottom. The
pyramids from the previous subsection are box diagrams. A {\em
frame} is a box diagram which is $S_m$-conjugate to a pyramid, where
$m$ is the number of rows in the pyramid.  Given a frame $F$ with
$m$ rows and $\sigma \in S_m$ we write $\sigma \cdot F$ for the
image of $F$ under the action of $\sigma$.

A frame is called {\em justified} if the boxes are aligned in
columns.
Given a frame $F$, the {\em left justification} of
$F$ is the frame $l(F)$ obtained from $F$ by left justifying the
rows.

A frame filled with complex numbers is called a {\em table}. Given a
table $A$, the {\em frame of $A$} is obtained by removing the
numbers in the boxes.  Let $F$ be a frame  with $m$ rows. We write
$\Tab(F)$ for the set of all tables with frame $F$.  For $1 \le i,j
\le m$, we write $A_i$ for the $i$th row of $A \in \Tab(F)$, and we
write $A^i_j$ for the table formed by rows $i$ to $j$ from $A$, for
$i < j$. For $A \in \Tab(F)$, we write $l(A) \in \Tab(l(F))$ for the
{\em left justification} of $A$.

For example,
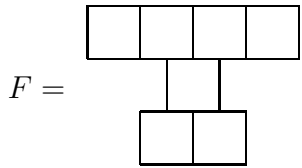
\[
    F =
    \begin{array}{c}
\begin{picture}(80,60)
\put(20,0){\line(1,0){40}}
\put(20,0){\line(0,1){20}}
\put(40,0){\line(0,1){20}}
\put(60,0){\line(0,1){20}}
\put(20,20){\line(1,0){40}}
\put(30,40){\line(1,0){20}}
\put(30,20){\line(0,1){20}}
\put(50,20){\line(0,1){20}}
\put(0,60){\line(1,0){80}}
\put(0,40){\line(1,0){80}}
\put(0,40){\line(0,1){20}}
\put(20,40){\line(0,1){20}}
\put(40,40){\line(0,1){20}}
\put(60,40){\line(0,1){20}}
\put(80,40){\line(0,1){20}}
\end{picture}
\end{array}
\]
is a frame,
\[
    A =
    \begin{array}{c}
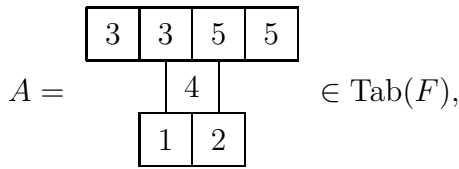
\begin{picture}(80,60)
\put(20,0){\line(1,0){40}}
\put(20,0){\line(0,1){20}}
\put(40,0){\line(0,1){20}}
\put(60,0){\line(0,1){20}}
\put(20,20){\line(1,0){40}}
\put(30,40){\line(1,0){20}}
\put(30,20){\line(0,1){20}}
\put(50,20){\line(0,1){20}}
\put(0,60){\line(1,0){80}}
\put(0,40){\line(1,0){80}}
\put(0,40){\line(0,1){20}}
\put(20,40){\line(0,1){20}}
\put(40,40){\line(0,1){20}}
\put(60,40){\line(0,1){20}}
\put(80,40){\line(0,1){20}}
\put(30,10){\makebox(0,0){{1}}}
\put(50,10){\makebox(0,0){{2}}}
\put(40,30){\makebox(0,0){{4}}}
\put(10,50){\makebox(0,0){{3}}}
\put(30,50){\makebox(0,0){{3}}}
\put(50,50){\makebox(0,0){{5}}}
\put(70,50){\makebox(0,0){{5}}}
\end{picture}
\end{array} \in \Tab(F),
\]
and
\[
    l(A) =
    \begin{array}{c}
   
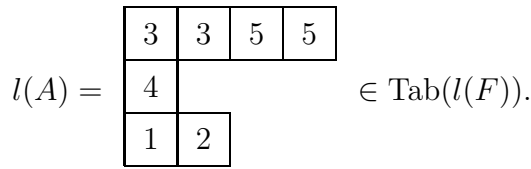
\begin{picture}(80,60)
\put(0,60){\line(1,0){80}}
\put(0,40){\line(1,0){80}}
\put(0,40){\line(0,1){20}}
\put(20,40){\line(0,1){20}}
\put(40,40){\line(0,1){20}}
\put(60,40){\line(0,1){20}}
\put(80,40){\line(0,1){20}}
\put(10,50){\makebox(0,0){{3}}}
\put(30,50){\makebox(0,0){{3}}}
\put(50,50){\makebox(0,0){{5}}}
\put(70,50){\makebox(0,0){{5}}}
\put(0,20){\line(1,0){40}}
\put(0,20){\line(0,1){20}}
\put(20,20){\line(0,1){20}}
\put(0,0){\line(1,0){40}}
\put(0,0){\line(0,1){20}}
\put(20,0){\line(0,1){20}}
\put(40,0){\line(0,1){20}}
\put(10,30){\makebox(0,0){{4}}}
\put(10,10){\makebox(0,0){{1}}}
\put(30,10){\makebox(0,0){{2}}}
\end{picture}
\end{array} \in \Tab(l(F)).
\]

Suppose $F$ is justified.  We say $A \in \Tab(F)$ is {\em column
strict} if the entries are strictly decreasing down columns.  The
{\em row equivalence class} of $A \in \Tab(F)$, denoted by $\bar A$,
is obtained by taking all possible permutations of the entries in
the rows of $A$.  We write $\Row(F)$ for the the set of row
equivalence classes of elements in $\Tab(F)$.

A {\em tableau} is a column strict left justified table $A$
such that the row lengths are weakly increasing from bottom to top,
and such that if $a$ lies to the left of $b$ in the same row of $A$,
then $a \not \geq b$.
The {\em shape} of a tableau $A$ with $m$ rows is the partition $\bp =
(p_1,\dots,p_m)$, where $p_i$ is the length of the $i$th row of
$A$.

Fix a frame $F$ with $m$ rows, let $1 \le k< m$ and write $s_k =
(k,k+1) \in S_m$.
An important notion for us is row swapping in tables, as defined in
\cite[\S4.3]{BroG}. We now define $s_k \star \,$ the row swapping
operation which takes as input  $\bar A \in \Row (F)$ and outputs
$s_k \star \bar A \in \Row (s_k \cdot F)$. Let $A$ be an element of
$\bar A$. First, if $\bar{l(A^k_{k+1})}$ does not contain en element
which is column strict then we say that $s_k \star \bar A$ is
undefined. Otherwise, let $c_1, c_2, \dots ,c_s$ be the entries of
$A_k$ and let $d_1,
 d_2, \dots, d_t$ be the entries of $A_{k+1}$.  We split into
two cases.

\noindent {\bf Case 1}: $s < t$.
We choose $e_1, \dots, e_s$ from $d_1, \dots d_t$ so that $e_i <
c_i$ and $\sum_{i=1}^s c_i-e_i$ is minimal. Then $e_1, \dots e_s$
form the entries of row $k+1$ in $s_k \star \bar A$, while the
remaining entries in $A_{k+1}$ are added to $c_1, \dots, c_s$ to
form the entries of row $k$ in $s_k \star \bar A$.

\noindent {\bf Case 2}: $s > t$.
We
choose $e_1, \dots, e_t$ from $c_1, \dots c_s$ so that $e_i
> d_i$ and $\sum_{i=1}^t e_i -d_i$ is minimal. Then $e_1,
\dots, e_t$ form the entries of row $k$ in $s_k \star A$, while the
remaining elements from row $k$ are added to $d_1, \dots, d_t$ to
form the entries of row $k+1$ in $s_k \star \bar A$.

In the example above we have
\[
\begin{array}{c}
\begin{picture}(80,60)
\put(30,60){\line(1,0){20}} \put(30,40){\line(0,1){20}}
\put(50,40){\line(0,1){20}} \put(40,0){\line(0,1){40}}
\put(60,0){\line(0,1){40}} \put(0,20){\line(0,1){20}}
\put(20,0){\line(0,1){40}} \put(40,0){\line(0,1){40}}
\put(60,0){\line(0,1){40}} \put(80,20){\line(0,1){20}}
\put(20,0){\line(1,0){40}} \put(0,20){\line(1,0){80}}
\put(0,40){\line(1,0){80}} \put(40,50){\makebox(0,0){{5}}}
\put(30,30){\makebox(0,0){{3}}} \put(50,30){\makebox(0,0){{4}}}
\put(10,30){\makebox(0,0){{3}}} \put(30,10){\makebox(0,0){{1}}}
\put(50,10){\makebox(0,0){{2}}} \put(70,30){\makebox(0,0){{5}}}
\end{picture}
\end{array}
\in s_1 \star \bar A.
\]

We finish this subsection with a brief discussion of the
Robinson--Schensted algorithm.  Given $A \in \Tab(F)$, we write
$\word(A)$ for the sequence of complex numbers created by listing
the entries in $A$ row by row from left to right, top to bottom.  In
the example above we have $ \word(A) = (5, 3, 3, 4, 5, 1, 2)$. The
Robinson--Schensted algorithm is a process that takes as input a
sequence of complex numbers and outputs a tableau. For a table $A$,
we write $\RS(A)$ for the output of the Robinson--Schensted
algorithm with input $\word(A)$. For $\bar A \in \Row(F)$ we write
$\RS(\bar A)$ to denote the row equivalence class of $\RS(A)$, where
$A \in \bar A$ is chosen so that if $a$ is to the left of $b$ in a
row of $A$, then $a \not \geq b$.
We refer the reader to \cite{F} or \cite[\S4.2]{BroG} for an
explanation of the Robinson--Schensted algorithm.

An important point for us is \cite[Lemma 4.8]{BroG}, which we recall
below.  In fact the lemma below is a little bit stronger than {\em
loc.\ cit.},
but is
straightforward to deduce.

\begin{Lemma} \label{L:RSswap}
Let $F$ be a frame with $m$ rows, $1 \le k \le m$, $\bar A \in
\Row(F)$ and $\bar A' \in \Row(s_k \cdot F)$ such that $s_k \star
\bar A$ is defined. Then $\RS(\bar A') = \RS(\bar A)$ if and only if
$\bar A' = s_k \star \bar A$.
\end{Lemma}

\subsection{Changing highest weight theories for $U(\g,e)$} \label{ss:changeA}
We use the notation from \S\ref{ss:classA}. In particular, $\bp =
(p_1 \ge p_2 \ge \dots \ge p_m)$ is a partition of $n$,  $P$ is a
pyramid with $m$ rows and row lengths given by the partition $\bp$,
and $K$ is the coordinate table of $P$.

For $\sigma \in S_m$, we recall that $\sigma \cdot P$ is the frame
obtained from $P$ by swapping rows according to $\sigma$; we define
$\sigma \cdot K$ similarly. For $K$ as in \eqref{e:coord} and
$\sigma = (123) \in S_3$, we have
$$
    \sigma \cdot K =
    \begin{array}{c}
\begin{picture}(80,60)
\put(10,0){\line(0,1){20}}
\put(0,40){\line(0,1){20}}
\put(30,20){\line(0,1){20}}
\put(50,20){\line(0,1){20}}
\put(20,40){\line(0,1){20}}
\put(40,40){\line(0,1){20}}
\put(30,0){\line(0,1){20}}
\put(50,0){\line(0,1){20}}
\put(60,40){\line(0,1){20}} \put(80,40){\line(0,1){20}}
\put(10,20){\line(1,0){40}} \put(10,0){\line(1,0){40}}
\put(0,40){\line(1,0){80}} \put(0,60){\line(1,0){80}}
\put(10,50){\makebox(0,0){{4}}} \put(30,50){\makebox(0,0){{5}}}
\put(50,50){\makebox(0,0){{6}}} \put(70,50){\makebox(0,0){{7}}}
\put(40,30){\makebox(0,0){{1}}} \put(20,10){\makebox(0,0){{2}}}
\put(40,10){\makebox(0,0){{3}}}
\end{picture}
\end{array}.
$$

Let $\sigma \in S_m$ and $i = 1,\dots,n$.  We write $\row_\sigma(i)
= \sigma(\row(i))$ for the row of $\sigma \cdot K$ that contains
$i$. We can define $e$, $\g(k)$, $\g_0$ and $\b_0$ from $\sigma
\cdot K$ in exactly the same way as we defined them from $K$. We
define
$$
\q_\sigma = \lan e_{i,j} \mid \row_\sigma(i) \le \row_\sigma(j)
\ran.
$$
Then $\q_\sigma$ is a parabolic subalgebra of $\g$ with Levi
subalgebra $\g_0$.  Moreover, it is easy to see that any parabolic
subalgebra of $\g$ with Levi subalgebra $\g_0$ occurs in this way
for some $\sigma \in S_m$. To shorten notation from now we write $\b
= \b_\q$ and $\b_\sigma = \b_{\q_\sigma}$.

For $B \in \Tab(\sigma \cdot F)$ we define $\lambda_{B,\sigma} =
\sum b_i \eps_i \in \t^*$, where $b_i$ is the number in the box of
$B$ which occupies the same position as $i$ in $\sigma \cdot K$. Let
$\Lambda_{B,\sigma}$ be the $W_0$-orbit of $\lambda_{B,\sigma}$. We
write $L_\sigma(\bar B)$ for the highest weight irreducible
$U(\g,e)$-module $L(\Lambda_{B,\sigma},\q_\sigma)$.  We define
$$
\Row^{+}(\sigma \cdot F) = \{\bar B \in \Row(\sigma \cdot F) \mid
\text{$L_\sigma(\bar B)$ is finite dimensional}\},
$$
and
$$
\cX^+(F) = \bigcup_{\sigma \in S_m} \Row^{+}(\sigma \cdot F).
$$

Below we state our theorem which tells us how to change between
different highest weight theories.  For statement we require the
{\em $\star$-action} of $S_m$ on $\cX^+(F)$. To define this let
$\sigma \in S_m$  and $\bar B \in \Row^+(\sigma \cdot F)$.  Write
$\sigma $ as a product of simple reflections $\sigma = s_{i_1}\dots
s_{i_l}$ and define
\begin{equation} \label{e:extend}
\sigma \star \bar B = s_{i_1} \star ( s_{i_2} \star ( \dots (s_{i_l}
\star \bar B) \dots )).
\end{equation}

\begin{Theorem} \label{T:hwA} $ $
\begin{enumerate}
\item[(i)]  The $\star$-action of $S_m$ on $\cX^+(P)$ is well defined.
\item[(ii)]
Let $\sigma,\tau \in S_m$, and $\bar B \in \Row^+(\sigma \cdot F)$,
$\bar B' \in \Row^+(\tau\sigma \cdot F)$. Then $L_\sigma(\bar B)
\iso L_{\tau\sigma}(\bar B')$ if and only if $\bar B' = \tau \star
\bar B$. \end{enumerate}
\end{Theorem}

Before proving the theorem we give a technical remark, which is
required in the proof.

\begin{Remark} \label{R:upsidedown}
An alternative proof of Theorem \ref{T:classfdA} now follows from
\cite[Proposition 3.12]{BroG} and the arguments in the proof of
\cite[Corollary 5.6]{BGK}.  This is based on an argument first
showing that $L(\bar A)$ is finite dimensional if and only if the
shape of $\RS(\bar A)$ is $\bp$ this requires \cite[Corollary
3.3]{Jo1}. Then it is an easy combinatorial argument to shows that
the shape of $\RS(A)$ is $\bp$ if and only if $\bar{l(A)}$
contains an element which is column strict. These arguments are also
valid, though the combinatorial argument is a bit more complicated,
if we use ``upside-down pyramids'', for which the row lengths are
decreasing from top to bottom, instead of pyramids.
\end{Remark}

\begin{proof}[Proof of Theorem \ref{T:hwA}]
First we have to give some more notation. For $1 \le k \le m$, we
define $t_k \in \t$ by $t_k = \sum_{j \mid \row(j) = k} e_{jj}$.
Then we have $\t^e = \lan t_1,\dots,t_m \ran$. Next for $1 \le k \ne
l \le m$ we define
$$
\s_{k,l} = \lan \{t_j \mid 1 \le j \le m, j \neq k,l\} \cup
\{t_k+t_l\} \ran.
$$
Then $\s_{k,l}$ is a full subalgebra of $\t^e$ and we have
$$
\g^{\s_{k,l}} = \lan e_{i,j} \mid \row(i) = \row(j) \text{ or }
\{\row(i),\row(j)\} = \{k,k+1\} \ran.
$$
Therefore,
$$
\g^{\s_k} \iso \left(\bigoplus_{j \neq k,l} \gl_{p_j}\right) \oplus
\gl_{p_k + p_{l}},
$$
and the finite $W$-algebra $U(\g^{\s_k},e)$ decomposes as a tensor
product
$$
U(\g^{\s_k},e) \iso \left( \bigotimes_{j \neq k,l}
U(\gl_{p_j},e_j)\right) \otimes U(\gl_{p_k + p_l},e_k + e_l),
$$
where $e_j$ is the projection of $e$ in $\gl_{p_j}$.

Now we show that for $\bar B \in \cX^+(F)$ and $1 \le k < m$, we
have that $s_k \star \bar B$ is defined. Let $B \in \bar B$ be such
that if $a$ is to the left of $b$ in a row of $B$ then $a \not \geq
b$. Let $\sigma \in S_m$ such that $\bar B \in \Row(\sigma \cdot F)$
and $L_\sigma(\bar B)$ is finite dimensional, and let $k' =
\sigma^{-1}(k)$ and $l = \sigma^{-1}(k+1)$. In this case
$L_\sigma(\bar B) = L(\Lambda_{B,\sigma},\q_\sigma)$ is finite
dimensional so as explained before \eqref{e:translevi}, we have
$L^{\s_{k',l}}(\Lambda_{B},\q_\sigma)$ is finite dimensional. Now
$L^{\s_{k',l}}(\Lambda_{B,\sigma},\q_\sigma)$ is tensor product of
irreducible highest weight modules for each of the finite
$W$-algebras in the tensor product decomposition of
$U(\g^{\s_{k',l}},e)$.  We consider the tensor factor corresponding
to $U(\gl_{p_{k'} + p_{l}},e_{k'} + e_{l})$. Under the natural
identifications, we see that up to some central shift (due to the
difference between ``$\rho$ for $\g$ and $\rho$ for $\gl_{p_{k'} +
p_{l}}$'') this tensor factor is the highest weight $U(\gl_{p_{k'} +
p_{l}},e_{k'} + e_l)$-module labelled by $B^k_{k+1}$; the central
shift corresponds to a constant being added to all the entries in
$B_k^{k+1}$. Thus by Theorem \ref{T:classfdA} and Remark
\ref{R:upsidedown} we have that $B^k_{k+1}$ is justified
row-equivalent to column strict.  Hence, $s_k \star \bar B$ is
defined.

Next we show that $L_{s_k\sigma}(s_k \star \bar B) \iso
L_\sigma(\bar B)$.  To do this we use Theorem \ref{T:hwgen} and a
result of Joseph which tells us when two highest weight
$U(\g)$-modules have the same annihilator. Given $\lambda =
\sum_{i=1}^n a_i \eps_i \in \t^*$, we define $\RS(\lambda)$ to be
the output of the Robinson--Schensted algorithm applied to
$\word(\lambda) = (a_1,\dots,a_n)$.  Then \cite[Th\'eor\`eme 1]{Jo1}
says that for $\lambda, \mu \in \t^*$ we have
\begin{equation} \label{e:jos}
    \Ann L(\lambda,\b) = \Ann L(\mu,\b)
    \text{ if and only if }
    \bar{\RS(\lambda)} = \bar{\RS(\mu)}.
\end{equation}
We have that $W = S_n$ acts on words of length $n$ and elements of
$\t^*$ in the usual way. Let $w_\sigma \in S_n$ be the permutation
such that $w_\sigma \cdot \word(K) = \word(\sigma \cdot K)$ and
define $w_{s_k\sigma}$ similarly. So we have $\b_\sigma = w_\sigma
\cdot \b$ and $\b_{s_k\sigma} = w_{s_k\sigma} \cdot \b$.

Let $s_k \star B$ denote an element in $s_k \star \bar B$
such that if $a, b$ lie in the same row as $s_k \star B$ then
$a \not \geq b$.
Now we have $\RS(\bar B) = \bar {\RS(w_\sigma^{-1} \cdot \lambda_{B,\sigma})}$
and $\RS(s_k \star \bar B) = \bar{\RS(w_{s_k\sigma}^{-1} \cdot \lambda_{s_k
\star B,s_k\sigma})}$.
By Lemma \ref{L:RSswap}, we have $\RS(\bar B) = \RS(s_k \star \bar
B)$, so we get $\bar{\RS(w_\sigma^{-1} \cdot
\lambda_{B,\sigma})} = \bar{\RS(w_{s_k\sigma}^{-1} \cdot
\lambda_{s_k \star B,s_k\sigma})}$, which implies that $\Ann
L(w_\sigma^{-1} \cdot \lambda_{B,\sigma},\b) = \Ann L(w_{s_k
\sigma}^{-1} \cdot \lambda_{s_k \star B, s_k\sigma},\b)$ by
\eqref{e:jos}. Hence, by Theorem \ref{T:hwgen} we have that
$[L_\sigma(\bar B)]$ and $[L_{s_k\sigma}(\bar B)]$ lie in the same
orbit of $C(e)$ in $\Irr_0U(\g,e)$.  We note that the condition
imposed on $B$ means that $\lan\lambda_{B'},\alpha^\vee\ran \not\in
\Z_{> 0}$ for all $\alpha \in \Phi_0^+$, and similarly for $s_k
\star B$, so that we can apply Theorem \ref{T:hwgen}. Since, $C(e)$
is trivial, we have $L_\sigma(\bar B) \iso L_{s_k\sigma}(s_k \star
\bar B)$.

Now let $\tau \in S_m$ then by writing $\tau$ as a product of simple
reflections we can define $\tau \star \bar B$ as in
\eqref{e:extend}.
By induction we have that each of the row swapping operations is
defined, and that
\begin{equation} \label{e:Aiso}
L_\sigma(\bar B) \iso L_{\tau\sigma}(\tau \star \bar B).
\end{equation}
Then we also see that $\tau \star \bar B$ does not depend on the
choice of the expression of $\tau$ in terms of simple reflection,
because $L_\sigma(\bar B) \iso L_{\tau\sigma}(\tau \star \bar B)$.
This means that the $\star$-action is a well defined action of $S_m$
on $\cX^+(F)$ giving (i).  Then (ii) is just \eqref{e:Aiso}.
\end{proof}

We state the following corollary, which is an immediate consequence
of Theorem \ref{T:hwA} and Lemma \ref{L:RSswap}.

\begin{Corollary}
Let $\sigma,\tau \in S_m$ and $\bar B \in \Row^+(\sigma \cdot F)$,
$\bar B' \in \Row^+(\tau\sigma \cdot F)$.  Then $L_\sigma(\bar B)
\iso L_{\tau\sigma}(\bar B)$ if and only if $\RS(\bar B) = \RS(\bar
B')$.
\end{Corollary}

As an example of the star-action, we give $\sigma \star \bar A$,
where
\[
A = \begin{array}{c}
\begin{picture}(60,80)
  \put(0,0){\line(1,0){60}}
  \put(0,20){\line(1,0){60}}
  \put(0,40){\line(1,0){60}}
  \put(20,60){\line(1,0){20}}
  \put(20,80){\line(1,0){20}}
  \put(0,0){\line(0,1){40}}
  \put(20,0){\line(0,1){80}}
  \put(40,0){\line(0,1){80}}
  \put(60,0){\line(0,1){40}}
   \put(28,10){\makebox(0,0){{-1}}}
   \put(30,30){\makebox(0,0){{1}}}
   \put(50,30){\makebox(0,0){{3}}}
   \put(8,30){\makebox(0,0){{-3}}}
   \put(50,10){\makebox(0,0){{2}}}
   \put(8,10){\makebox(0,0){{-4}}}
   \put(28,50){\makebox(0,0){{-2}}}
   \put(30,70){\makebox(0,0){{4}}}
   \end{picture}
   \end{array}
   \]
and $\sigma = (123)$.  Then we have
\[
\begin{array}{c}
\begin{picture}(60,80)
  \put(0,0){\line(1,0){60}}
  \put(0,20){\line(1,0){60}}
  \put(20,40){\line(1,0){20}}
  \put(0,60){\line(1,0){60}}
  \put(0,80){\line(1,0){60}}
  \put(0,0){\line(0,1){20}}
  \put(0,60){\line(0,1){20}}
  \put(20,0){\line(0,1){80}}
  \put(40,0){\line(0,1){80}}
  \put(60,0){\line(0,1){20}}
  \put(60,60){\line(0,1){20}}
   \put(8,10){\makebox(0,0){{-4}}}
   \put(28,10){\makebox(0,0){{-1}}}
   \put(50,10){\makebox(0,0){{2}}}
   \put(28,30){\makebox(0,0){{-3}}}
   \put(30,50){\makebox(0,0){{3}}}
   \put(8,70){\makebox(0,0){{-2}}}
   \put(30,70){\makebox(0,0){{1}}}
   \put(50,70){\makebox(0,0){{4}}}
   \end{picture}
   \end{array} \in \sigma \star \bar A.
\]
Note that $A$ is row equivalent to column strict, yet $\sigma \star
\bar A$ does not contain any column strict elements.  So for
different highest weight theories the classification of finite
dimensional irreducible $U(\g,e)$-module is not just that
$L_\sigma(\bar A)$ is finite dimensional if and only if $\bar A$
contains a column strict table.

\begin{Remark} \label{R:restA}
Let $\sigma, \tau \in S_m$ and suppose that $\tau\sigma \cdot F =
\sigma \cdot F$, i.e.\ $\tau$ permutes rows $\sigma \cdot F$ of the
same length.  Then as explained in \cite[Section 6]{BruG}, there
exists an element of the restricted Weyl group $w \in W^e$ such that
$w \cdot \q_\sigma = \q_{\tau\sigma}$.  Thus for any $\bar B \in
\Row^{+}(\sigma \cdot F)$, we have $L_\sigma(\bar B) \iso
L_{\tau\sigma}(\bar B)$ by Proposition \ref{P:restchange}.  This can
also be easily verified by noting that a row swapping operation on
two rows of the same length is trivial.
\end{Remark}

\section{Changing highest weight theories associated to
even multiplicity finite $W$-algebras} \label{S:evenchange}

In this section we prove Theorem \ref{T:changehwevenmult}, which
tells us how to pass between different highest weight theories when
$\g$ is of type C or D and $e$ is even multiplicity.  We recall some
definitions from \cite[Section 4]{BroG} regarding s-frames and
s-tables in \S\ref{ss:sframe}. Then we use s-tables to give the
notation for finite $W$-algebras in \S\ref{ss:setupclass} and the
combinatorics for the description of the highest weight theories in
\S\ref{ss:classhwthy}. In the remaining subsections we review the
classification of finite dimensional $U(\g,e)$-modules from
\cite[Section 5]{BroG} and describe the bijection between
parameterizing sets for different highest weight theories.

In this section we often consider sets of the form
$\{1,2,\dots,l,-l,\dots,-2,-1\}$ and we use the unconventional total
order on this set given by  $1 \le 2 \le \dots \le l \le -l \le
\dots \le -2 \le -1$.

\subsection{s-frames and s-tables} \label{ss:sframe}

The combinatorics for the highest weight theories for finite
$W$-algebras associated to even multiplicity nilpotent elements in
classical Lie algebra algebras involves a skew-symmetric version of
tables
called s-tables.  Below we
review the terminology for s-frames and s-tables from
\cite[\S4.4]{BroG}.

We define an {\em s-frame} to be a frame where the boxes, are
arranged symmetrically around the origin. We say that an s-frame is
a {\em symmetric pyramid} if the row lengths weakly decrease from the
centre outwards; we note that a symmetric pyramid is uniquely
determined by its row lengths. In this paper we only consider
s-frames which have an even number of rows.

An example of an s-frame (which is not a symmetric pyramid) is
$$
    \begin{array}{c}
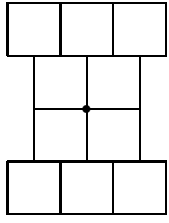
\begin{picture}(60,80)
\put(0,0){\line(1,0){60}} \put(0,20){\line(1,0){60}}
\put(10,40){\line(1,0){40}} \put(0,60){\line(1,0){60}}
\put(0,80){\line(1,0){60}} \put(0,0){\line(0,1){20}}
\put(20,0){\line(0,1){20}} \put(40,0){\line(0,1){20}}
\put(60,0){\line(0,1){20}} \put(10,20){\line(0,1){40}}
\put(30,20){\line(0,1){40}} \put(50,20){\line(0,1){40}}
\put(0,60){\line(0,1){20}} \put(20,60){\line(0,1){20}}
\put(40,60){\line(0,1){20}} \put(60,60){\line(0,1){20}}
\put(30,40){\circle*{3}}
\end{picture}
\end{array}.
$$

We define an {\em s-table} to be an s-frame for which every box is
filled with a complex number. Furthermore we require that the boxes
be filled skew-symmetrically with respect to the centre. Given an
$s$-frame $F$, we write $\sTab(F)$ for the set of s-tables with
frame $F$. We write $\bar A^s = \bar A \cap \sTab(F)$ for the set of
s-tables row equivalent to $A$.  For example
\begin{equation} \label{exstable}
\begin{array}{c}
\begin{picture}(80,80)
\put(20,0){\line(1,0){40}} \put(0,20){\line(1,0){80}}
\put(0,40){\line(1,0){80}} \put(0,60){\line(1,0){80}}
\put(20,80){\line(1,0){40}} \put(20,0){\line(0,1){20}}
\put(40,0){\line(0,1){20}} \put(60,0){\line(0,1){20}}
\put(60,20){\line(0,1){20}} \put(80,20){\line(0,1){20}}
\put(60,40){\line(0,1){20}} \put(80,40){\line(0,1){20}}
\put(20,80){\line(0,-1){20}} \put(40,80){\line(0,-1){20}}
\put(60,80){\line(0,-1){20}} \put(0,20){\line(0,1){40}}
\put(20,20){\line(0,1){40}} \put(40,20){\line(0,1){40}}
\put(30,70){\makebox(0,0){{-7}}} \put(50,70){\makebox(0,0){{3}}}
\put(10,50){\makebox(0,0){{-8}}} \put(30,50){\makebox(0,0){{-4}}}
\put(50,50){\makebox(0,0){{2}}} \put(70,50){\makebox(0,0){{5}}}
\put(8,30){\makebox(0,0){{-5}}} \put(28,30){\makebox(0,0){{-2}}}
\put(48,30){\makebox(0,0){{4}}} \put(68,30){\makebox(0,0){{8}}}
\put(28,10){\makebox(0,0){{-3}}} \put(48,10){\makebox(0,0){{7}}}
\put(40,40){\circle*{3}}
\end{picture}
\end{array}
 \in \sTab^\leq(F),
\end{equation}
where $F$ is its s-frame.   A piece of notation that we require
later is as follows. Given a sign $\phi \in \{\pm\}$, we define
\[
  \sTab_\phi(F) =
  \begin{cases}
      \{A \in \sTab(F) \mid A \text{ has all entries in $\Z$ or all entries in $\frac{1}{2} + \Z$}\} & \text{if $\phi = +$} ; \\
      \{ A \in \sTab(F) \mid  \text{$A$ has all entries in $\Z$}\} & \text{if $\phi =
-$}.
  \end{cases}
\]
The subset of $\sTab_\phi(F)$ consisting of s-tables with entries
weakly increasing along rows is denoted by $\sTab_\phi^\leq(F)$.

Let $F$ be an s-frame and $A \in \sTab(F)$.
By assumption, $F$ has an even number of rows, say $2m$. We label
the rows of $F$ and $A$ with $1, \dots, m, -m, \dots, -1$ from
bottom to top. Given $i = \pm 1,\dots,\pm m$ we write $A_i$ for row
of $A$ labelled by $i$, and for $i > 0$ we write $A_{-i}^i$ for the
s-table obtained by removing rows $\pm 1,\dots, \pm (i-1)$. The
table obtained from $A$ by removing all boxes below the central
point is denoted by $A_+$. For example if $A$ is the table above,
then
$$
    A_+ =
    \begin{array}{c}
\begin{picture}(80,40)
\put(0,0){\line(0,1){20}} \put(20,0){\line(0,1){40}}
\put(40,0){\line(0,1){40}} \put(60,0){\line(0,1){40}}
\put(80,0){\line(0,1){20}} \put(0,0){\line(1,0){80}}
\put(0,20){\line(1,0){80}} \put(20,40){\line(1,0){40}}
\put(30,30){\makebox(0,0){{-7}}} \put(50,30){\makebox(0,0){{3}}}
\put(10,10){\makebox(0,0){{-8}}} \put(30,10){\makebox(0,0){{-4}}}
\put(50,10){\makebox(0,0){{2}}} \put(70,10){\makebox(0,0){{5}}}
\end{picture}
\end{array}
.
$$

Finally in this subsection we generalize the row swapping procedure
to s-tables. As above let $F$ be an s-frame with $2m$ rows, and let
$A \in \sTab_\phi^\leq(F)$, where $\phi \in \{\pm\}$. Let $k =
1,\dots,m-1$. Then we define
$$
\bar s_k \star A = s_{-k} \star (s_k \star A),
$$
where $s_k$ swaps rows $k$ and $k+1$ as defined in
\S\ref{ss:frames}, and $s_{-k}$ swaps rows $-(k+1)$ and $-k$ using
the same rules. Here we define the row swapping operations directly
on elements of $\sTab_\phi^\leq(F)$ rather than or row equivalence
classes, because there is a unique element of $\sTab_\phi^\leq(F)$
in any row equivalence class. We note that $s_k \star A$ is defined
if and only if $s_{-k} \star A$ is defined, and that the operators
$s_k$ and $s_{-k}$ commute.
Also we note that when $s_k$ is defined, the action of $s_{-k}$ is
``dual'' to that of $s_k$, so $\bar s_k \star A$ is an s-table.

\subsection{Notation for even multiplicity finite $W$-algebras} \label{ss:setupclass}

For the rest of this section, we fix a sign $\phi \in \{\pm\}$.  As
a shorthand we say that an integer $l$ is $\phi$-even if $\phi = +$
and $l$ is even or $\phi = -$ and $l$ is odd; we define $\phi$-odd
similarly.

We specify coordinates for $\sp_{2n}$ and $\so_{2n}$. Let $V =
\C^{2n}$ be the $2n$-dimensional vector space with standard basis
$\{e_1,\dots,e_n,e_{-n},\dots,e_{-1}\}$ and nondegenerate bilinear
form $(\cdot,\cdot)$ defined by $(e_i,e_j) = 0$ if $i$ and $j$ have
the same sign, and $(e_i,e_{-j}) = \delta_{i,j}$, $(e_{-i},e_j) =
\phi \delta_{i,j}$ for  $i,j = 1,\dots,n$.  Let $\tilde G =
G^\phi_{2n} = \{x \in \GL_{2n} \mid (xv|xv') = (v|v') \text{ for all
} v,v' \in V\}$, and $\g = \g_{2n}^\phi = \{x \in \gl_{2n} \mid
(xv|v') = -(v|xv') \text{ for all } v,v' \in V\}$ be the Lie algebra
of $\tilde G$. So $\tilde G = \operatorname{O}_{2n}$ and $\g =
\so_{2n}$ if $\phi = +$, and $\tilde G = \Sp_{2n}$ and $\g =
\sp_{2n}$ if $\phi = -$. We write $G$ for the identity component
group of $\tilde G$, so $G = \tilde G$ in the type C case, and $G =
\SO_{2n}$ in the type D case.  We let $(\cdot | \cdot)$ be the trace
form on $\g$.

Let $\{e_{i,j} \mid i,j = 1,\dots,n,-n,\dots,-1\}$ be the standard
basis of $\gl_{2n}$, and define $f_{i,j} = e_{i,j} - \eta_{i,j}
e_{-j,-i}$ where $\eta_{i,j} = 1$ if $i$ and $j$ have the same sign
and $\eta_{i,j} = \phi$ if $i$ and $j$ have different signs.  Then
the standard basis of $\g$ is $\{ f_{i,j} \mid i  < -j \}$ if $\phi
= +$ and $\{ f_{i,j} \mid i \leq -j\}$ if $\phi = -$, where we use
the order given by $1 \le 2 \le \dots \le n \le -n \le \dots \le -2
\le -1$ . Let $\t = \lan f_{i,i} \mid i = 1,\dots,n \ran$ be the
standard Cartan subalgebra of $\g$ of diagonal matrices. We define
$\{\eps_i \mid i = 1,\dots,n\}$ to be the basis of $\t^*$ dual to
$\{f_{i,i} \mid i = 1,\dots,n\}$.

We recall that nilpotent $\tilde G$-orbits in $\g$ are parameterized by
partitions $\bp$, such that each $\phi$-even part of $\bp$ has even
multiplicity when $\g = \so_{2n}$.  For $\g = \so_{2n}$,
we also recall that a nilpotent $\tilde G$-orbit parameterized by $\bp$ is a
single $G$-orbit unless all parts of $\bp$ are even and of even
multiplicity.  In this latter case, where we say that $\bp$ is {\em
very even}, the $\tilde G$-orbit parameterized by $\bp$ splits into two
$G$-orbits.

We recall the structure of the component group $\tilde C(e)$ of the
centralizer of $e$ in $\tilde G$.  Suppose $e \in \g$ lies in the
nilpotent $\tilde G$-orbit corresponding to the partition $\bp$.
Then $\tilde C(e) \iso \Z_2^d$, where $d$ is the number of distinct
$\phi$-odd parts of $\bp$, see for example \cite[\S3.13]{Ja}.
We note that $C(e)$ is equal to $\tilde C(e)$
unless $\g = \so_{2n}$ and $\bp$ has an odd part, in which case $C(e)$
has index 2 in $\tilde C(e)$.

For the remainder of this section we fix an even multiplicity
partition $\bp = (p_1^2,\dots,p_r^2)$ of $2n$, where $p_i \ge
p_{i+1}$ for each $i$. The {\em symmetric pyramid} of $\bp$ is the
symmetric pyramid with row lengths given by $\bp$ as defined in
\S\ref{ss:sframe}; we write $P = P_\bp$ for this s-frame. The table
with frame $P$ and with boxes filled by $1,\dots,n,-n,\dots,-1$ from
left to right and top to bottom is called the {\em coordinate
pyramid of $\bp$} and denoted by $K = K_\bp$.  For example
$$
K =
\begin{array}{c}
\begin{picture}(60,80)
\put(10,0){\line(1,0){40}} \put(0,20){\line(1,0){60}}
\put(0,40){\line(1,0){60}} \put(0,60){\line(1,0){60}}
\put(10,80){\line(1,0){40}} \put(10,0){\line(0,1){20}}
\put(30,0){\line(0,1){20}} \put(50,0){\line(0,1){20}}
\put(0,20){\line(0,1){40}} \put(20,20){\line(0,1){40}}
\put(40,20){\line(0,1){40}} \put(60,20){\line(0,1){40}}
\put(10,60){\line(0,1){20}} \put(30,60){\line(0,1){20}}
\put(50,60){\line(0,1){20}}
\put(30,40){\circle*{3}} \put(10,50){\makebox(0,0){{3}}}
\put(30,50){\makebox(0,0){{4}}} \put(50,50){\makebox(0,0){{5}}}
\put(20,70){\makebox(0,0){{1}}} \put(40,70){\makebox(0,0){{2}}}
\put(18,10){\makebox(0,0){{-2}}} \put(38,10){\makebox(0,0){{-1}}}
\put(8,30){\makebox(0,0){{-5}}} \put(28,30){\makebox(0,0){{-4}}}
\put(48,30){\makebox(0,0){{-3}}}
\end{picture}
\end{array}
$$
is a coordinate table.

We define the nilpotent element $e \in \g$ with Jordan type $\bp$ by
$e = \sum f_{i,j}$, where we sum over all $i,j$ such that $i$ and
$j$ are positive and $j$ is in the box immediately to the right of
$i$ in $K$.  We write $\col(i)$ for the $x$-coordinate of the box in
$K$ containing $i$ and we define $h  = \sum_{i=1}^n - \col(i)
f_{i,i}$. For example, if $K$ is as above, we have $e = f_{1,2} +
f_{3,4} + f_{4,5}$ and $h = -f_{1,1} + f_{2,2} - 2f_{3,3}+2f_{5,5}$.
Then the $\ad h$ eigenspace decomposition gives the Dynkin grading
$$
\g(k) = \lan f_{i,j} \mid \col(j) - \col(i) = k \ran.
$$
The finite $W$-algebra $U(\g,e)$ can now be
defined as in \S\ref{ss:Wdef}.

We do not consider other good gradings for $e$ here, as there are
not many non-Dynkin good gradings, so it is not particularly
advantageous to do so; we refer the reader to \cite[Sections 5 and
6]{EK} and \cite[Sections 6 and 7]{BruG} for more information on
good gradings for classical Lie algebras.

\subsection{Highest weight theories for $U(\g,e)$} \label{ss:classhwthy}

We now discuss highest weight theories for $U(\g,e)$.  We continue
to use the notation from the previous subsection; in particular, $P$
is the symmetric pyramid of $\bp$ and $K$ is the coordinate pyramid
of $\bp$.  First we consider the highest weight theory for a
particular choice $\q$ of parabolic subalgebra, then we give the
notation for other choices of parabolic subalgebra.

For $i = \pm 1,\dots,\pm n$ we write $\row(i)$ for the row of $K$ in
which $i$ appears; recall that rows in $P$ are labelled with
$-m\dots,-1,1,\dots,m$ from bottom to top. Then we have
$$
\g_0 = \lan f_{i,j} \mid \row(i) = \row(j) \ran,
$$
and
$$
\b_0 = \lan f_{i,j} \mid \row(i) = \row(j) \text{ and } \col(i) \le
\col(j) \ran
$$
Let
$$
\q = \lan f_{i,j} \mid \row(i) \le \row(j) \ran,
$$
which is a parabolic subalgebra of $\g$ with Levi subalgebra $\g_0$;
here we are using the ordering $1 \le 2 \le \dots \le m \le -m \le
\dots \le -2 \le -1$.

To each $A \in \sTab_\phi^{\le}(F)$ we associate a weight $\lambda_A
= \sum a_i \eps_i \in \t^*$, where $a_i$ is the number in the box of
$A$ which occupies the same position as $i$ in $K$.  For example,
with $K$ as above and
\begin{equation} \label{e:stableex}
  A =  \begin{array}{c}
\begin{picture}(60,80)
\put(10,0){\line(1,0){40}} \put(0,20){\line(1,0){60}}
\put(0,40){\line(1,0){60}} \put(0,60){\line(1,0){60}}
\put(10,80){\line(1,0){40}} \put(10,0){\line(0,1){20}}
\put(30,0){\line(0,1){20}} \put(50,0){\line(0,1){20}}
\put(0,20){\line(0,1){40}} \put(20,20){\line(0,1){40}}
\put(40,20){\line(0,1){40}} \put(60,20){\line(0,1){40}}
\put(10,60){\line(0,1){20}} \put(30,60){\line(0,1){20}}
\put(50,60){\line(0,1){20}}
\put(30,40){\circle*{3}} \put(8,50){\makebox(0,0){{-3}}}
\put(30,50){\makebox(0,0){{1}}} \put(50,50){\makebox(0,0){{4}}}
\put(20,70){\makebox(0,0){{2}}} \put(40,70){\makebox(0,0){{7}}}
\put(18,10){\makebox(0,0){{-7}}} \put(38,10){\makebox(0,0){{-2}}}
\put(8,30){\makebox(0,0){{-4}}} \put(28,30){\makebox(0,0){{-1}}}
\put(50,30){\makebox(0,0){{3}}}
\end{picture}
\end{array},
\end{equation}
we have
$$
\lambda_A =  -3\eps_1 + \eps_2+ 4\eps_3 + 2\eps_4 + 7\eps_5.
$$
Let $\Lambda_A$ be the $W_0$-orbit of $\lambda_A$.  We note that
$W_0$ is isomorphic to $S_{p_1} \times \dots S_{p_m}$ and the action
of $W_0$ on $\t^*$ corresponds to $W_0$ acting on tables by
permuting entries in rows. Thus $\Lambda_A$ corresponds to the row
equivalence class $\bar A^s$ of $A$. We write $L(A)$ for the highest
weight irreducible $U(\g,e)$-module $L(\Lambda_A,\q)$, as defined in
\S\ref{ss:hwthy}. Later, in Theorem \ref{T:BroGmain}, we state the
main theorem from \cite{BroG}, which determines when $L(A)$ is
finite dimensional.

We note that the restriction to tables in $\sTab_\phi(P)$
corresponds to the central character of $L(A)$ being integral. Also
as we use tables in $\sTab_\phi^\le(P)$ there is no need to use the
row equivalence class in the notation for $L(A)$.

We now give the notation for highest weight theories corresponding
to other choices of parabolic subalgebra. Let $W_m$ denote the Weyl
group of type $B_m$ acting on $\{\pm 1,\dots,\pm m\}$ in the usual
way.
We write $\bar S_m$ for the subgroup of $W_m$ isomorphic to $S_m$
consisting of the permutations with no sign changes.  The standard
generators of $W_m$ are denoted by $r,\bar s_1,\dots,\bar s_{n-1}$,
where $r$ is the transposition $(n,-n)$ and $\bar s_1,\dots,\bar
s_{n-1}$ are the standard generators of $\bar S_m$, so $\bar s_i =
(i,i+1)(-i,-i-1)$. Given $\sigma \in \bar S_m$ we write $\underline
\sigma$ for the corresponding element of $S_m$.

For $\sigma \in W_m$, we define $\sigma \cdot P$ to be the frame
obtained from $P$ by permuting rows according to $\sigma$ and define
$\sigma \cdot K$ similarly. For example for $K$ as in \eqref{e:coord} and
$\sigma = (1,-2)(2,-1) \in W_2$, we have
$$
    \sigma \cdot K =
    \begin{array}{c}
\begin{picture}(60,80)
\put(0,0){\line(1,0){60}} \put(0,20){\line(1,0){60}}
\put(10,40){\line(1,0){40}} \put(0,60){\line(1,0){60}}
\put(0,80){\line(1,0){60}} \put(0,0){\line(0,1){20}}
\put(20,0){\line(0,1){20}} \put(40,0){\line(0,1){20}}
\put(60,0){\line(0,1){20}} \put(10,20){\line(0,1){40}}
\put(30,20){\line(0,1){40}} \put(50,20){\line(0,1){40}}
\put(0,60){\line(0,1){20}} \put(20,60){\line(0,1){20}}
\put(40,60){\line(0,1){20}} \put(60,60){\line(0,1){20}}
\put(30,40){\circle*{3}} \put(8,70){\makebox(0,0){{-5}}}
\put(28,70){\makebox(0,0){{-4}}} \put(48,70){\makebox(0,0){{-3}}}
\put(18,50){\makebox(0,0){{-2}}} \put(38,50){\makebox(0,0){{-1}}}
\put(20,30){\makebox(0,0){{1}}} \put(40,30){\makebox(0,0){{2}}}
\put(10,10){\makebox(0,0){{3}}} \put(30,10){\makebox(0,0){{4}}}
\put(50,10){\makebox(0,0){{5}}}
\end{picture}
\end{array}
$$

Let $\sigma \in W_m$ and $i = \pm 1,\dots, \pm n$.  We write $\row_\sigma(i)$ for the row of
$\sigma \cdot K$ that contains $i$. We can define $e$, $\g(k)$,
$\g_0$ and $\b_0$ from $\sigma \cdot K$ in exactly the same way as
we defined them from $K$. We define
$$
\q_\sigma = \lan e_{i,j} \mid \row_\sigma(i) \le \row_\sigma(j)
\ran.
$$
Then $\q_\sigma$ is a parabolic subalgebra of $\g$ with Levi
subalgebra $\g_0$.  Moreover, it is easy to see that any parabolic subalgebra of $\g$ with
Levi subalgebra $\g_0$ occurs in this way for some $\sigma \in W_m$.

For $B \in \sTab_\phi^\le(\sigma \cdot P)$ we define
$\lambda_{B,\sigma} = \sum b_i \eps_i \in \t^*$, where $b_i$ is the
number in the box of $B$ which occupies the same position as $i$ in
$\sigma \cdot K$. Let $\Lambda_{B,\sigma}$ be the $W_0$-orbit of
$\lambda_{B,\sigma}$. We write $L_\sigma(B)$ for the highest weight
irreducible $U(\g,e)$-module $L(\Lambda_{B,\sigma},\q_\sigma)$.  We
define
$$
\sTab_\phi^{+}(\sigma \cdot P) = \{B \in \sTab_\phi^{\le}(\sigma
\cdot P) \mid \text{$L_\sigma(B)$ is finite dimensional}\}
$$
and
$$
\cX_\phi^+(P) = \bigcup_{\sigma \in W_m} \sTab_\phi^{+}(\sigma \cdot
P).
$$

\subsection{Changing the highest weight theory ``in the top half''}

In this section we begin to show how to pass between different
highest weight theories, where the change involves permuting rows
according to an element of $\bar S_m \sub W_m$, i.e.\ permuting rows
in the top half.  In the statement we use the $\star$-action of $\bar
S_m$ on $\cX_\phi^+(P)$ defined by extending the row swapping
operations $\bar s_k$ from \S\ref{ss:sframe} in analogy to
\eqref{e:extend}.

\begin{Proposition} \label{P:tophalf} $ $
\begin{enumerate}
\item[(i)]The $\star$-action of $\bar S_m$ on $\cX_\phi^+(P)$ is well defined.
\item[(ii)]
Let $\sigma \in W_m$, $\tau \in \bar S_m$, and $B \in
\sTab^+_\phi(\sigma \cdot P)$, $B' \in \sTab^+_\phi(\tau\sigma \cdot
F)$.
Then $L_\sigma(B) \iso L_{\tau\sigma}(B')$ if
and only if $B' = \tau \star B$.
\end{enumerate}
\end{Proposition}

\begin{proof}
First let $\sigma, \tau \in \bar S_m$, and $B \in
\sTab^+_\phi(\sigma \cdot P)$.
We set $t = \sum_{\row_\sigma(i) > 0} f_{i,i}$
and
$\s = \lan t \ran$, which is a full subalgebra of $\t^e$.
Then $\g^\s \iso \gl_{n}$
and we see that $(\q_\sigma)_\s = (\q_{\tau\sigma})_\s$.

Now $L_\sigma(B) = L(\Lambda_{B},\q_\sigma)$ is finite dimensional so as explained before
\eqref{e:translevi}, we have $L^{\s}(\Lambda_{B},\q_\sigma)$ is a finite
dimensional $U(\g^s,e)$-module.  We see that up to some
central shift (due to the different root systems for $\g$ and
and $\gl_n$) $L^{\s}(\Lambda_{B},\q_\sigma)$ is isomorphic to the
highest weight $U(\gl_n,e)$-module $L_{\underline \sigma}(B_+)$; this central shift corresponds
to adding a constant to all entries in $B_+$.

Now using Theorem \ref{T:hwA}, the table $\underline \tau \star B_+$
is defined. Clearly $\tau \star B$ is the s-table with $(\tau
\star B)_+ = \underline \tau \star B_+$, so, in particular, it is well defined. Again by Theorem
\ref{T:hwA}, we have $L_{\underline \sigma}(B_+) \iso L_{\underline
\tau \, \underline \sigma}(\underline \tau \star B_+)$, which
implies that $L^{\s}(\Lambda_{B},\q_\sigma) \iso
L^{\s}(\Lambda_{\tau \star B},\q_{\tau\sigma})$.  Thus as
$(\q_\sigma)_\s = (\q_{\tau\sigma})_\s$, we get that
$M_\s(L^{\s}(\Lambda_{B},\q_\sigma)) \iso M_\s(L^{\s}(\Lambda_{\tau
\star B},\q_{\tau\sigma}))$, which means that $L_\sigma(B) \iso
L_{\tau\sigma}(\tau \star B)$.
In particular, $L_{\tau\sigma}(\tau \star B)$ is finite dimensional.

In the case that $\sigma \in W_m \setminus \bar S_m$,  all of
the arguments above go through with a minor complication regarding the
identification $\g^\s \iso \gl_n$.
\end{proof}

\subsection{The component group action} \label{ss:comp}
In order to state the classification of finite dimensional
irreducible $U(\g,e)$-modules in Theorem \ref{T:BroGmain} we need to
recall the action of the component group $\tilde C(e)$ on
$\cX^+_\phi(P)$ from \cite[\S5.3]{BroG}.   In fact we complete the
verification that we do get the true action of $\tilde C(e)$, see
\cite[Remark 5.9]{BroG}. The component group action is also required
for Theorem \ref{T:changehwevenmult}, where we complete the
description of how to pass between different highest weight
theories.

The description of the action depends on the notion of the
{\em $\sharp$-element} of a list of complex numbers.
Given a list $(a_1,\dots,a_{2k+1})$ of complex numbers
let $\{(a_1^{(i)},\dots,a_{2k+1}^{(i)}) \mid i \in I\}$ be the set
of all permutations of this list which satisfy
$a^{(i)}_{2j-1}+a^{(i)}_{2j} > 0$
for each $j=1,\dots,k$. Assuming that such rearrangements exist, we
define the {\em $\sharp$-element} of $(a_1,\dots,a_{2k+1})$ to be
the unique maximal element of the set $\{a_{2k+1}^{(i)} \mid i \in
I\}$. On the other hand, if no such rearrangements exist, we say
that the $\sharp$-element of $(a_1,\dots,a_{2k+1})$ is undefined.
For example, the $\sharp$-element of $(-3, -1, 2)$ is $-3$, whereas
the $\sharp$-element of $(-3,-2,1)$ is undefined.
We abuse notation somewhat by saying that the $\sharp$-element of
a list of numbers with an even number of elements is the
$\sharp$-element of that list with $0$ inserted.

We begin by considering the case where $\bp = (n^2)$, $n$
is $\phi$-odd, and $n$ is odd if $\g = \so_{2n}$. In this case we
have $\tilde C(e) \iso \Z_2 = \lan c \ran$, and we define an
operation of $c$ on $\sTab_\phi^+(P)$ as follows. Let $A \in
\sTab^+_\phi(P)$ and let $a_1, \dots, a_n$ be the entries of row $1$ of $A$. By
\cite[Theorem 1.2]{Bro2} the $\sharp$-element of $a_1, \dots, a_n$
is defined; let $a$ be this number. We declare that $c \cdot A \in
\sTab^\leq(P)$ is the s-table obtained from $A$ by replacing one
occurrence of $a$ in row $1$ with $-a$, and one occurrence of $-a$
in row $-1$ with $a$. Then \cite[Theorem 1.3]{Bro2} says that $c
\cdot L(A) = L(c \cdot A)$; in particular, $c \cdot A \in
\sTab^+_\phi(P)$

An
example of this action is
$$
    c \cdot
    \begin{array}{c}
\begin{picture}(40,40)
\put(0,0){\line(1,0){40}}
\put(0,20){\line(1,0){40}} \put(0,40){\line(1,0){40}}
\put(0,0){\line(0,1){20}} \put(20,0){\line(0,1){20}}
\put(40,0){\line(0,1){20}} \put(0,20){\line(0,1){20}}
\put(20,20){\line(0,1){20}} \put(40,20){\line(0,1){20}}
\put(8,10){\makebox(0,0){{-2}}} \put(28,10){\makebox(0,0){{-1}}}
\put(10,30){\makebox(0,0){{1}}} \put(30,30){\makebox(0,0){{2}}}
\put(20,20){\circle*{3}}
\end{picture}
\end{array}
=
    \begin{array}{c}
\begin{picture}(40,40)
\put(0,0){\line(1,0){40}}
\put(0,20){\line(1,0){40}} \put(0,40){\line(1,0){40}}
\put(0,0){\line(0,1){20}} \put(20,0){\line(0,1){20}}
\put(40,0){\line(0,1){20}} \put(0,20){\line(0,1){20}}
\put(20,20){\line(0,1){20}} \put(40,20){\line(0,1){20}}
\put(8,10){\makebox(0,0){{-1}}} \put(30,10){\makebox(0,0){{2}}}
\put(8,30){\makebox(0,0){{-2}}} \put(30,30){\makebox(0,0){{1}}}
\put(20,20){\circle*{3}}
\end{picture}
\end{array}
.
$$

Now we define an operation of $c$ on $\cX^+(P)$ for $\bp$ any even
multiplicity partition. Let $B \in \cX^+(P)$ and let $\sigma \in
W_m$ such that $B \in \sTab^+_\phi(\sigma \cdot P)$. Suppose that
the length of row $m$ in $B$ is $\phi$-even, then we define $c \cdot
B = B$. Next suppose that the length of row $m$ in $B$ is
$\phi$-odd. Below we justify that $c \cdot B_{-m}^m$ is defined.
This allows us to define $c \cdot B$ to be the table obtained from
$B$ by replacing rows $m$ and $-m$ by $c \cdot B_{-m}^m$.

To justify that $c \cdot B_{-m}^m$ is defined we let $t_\sigma =
\sum_{i \mid \row_\sigma(i) \neq m} f_{i,i}$ and $\s_\sigma = \lan
t_\sigma \ran$, which is a full subalgebra of $\t^e$.  The Levi
subalgebra $\g^{\s_\sigma}$ is isomorphic to $\gl_{n-p_{\sigma(m)}}
\oplus \g^\phi_{2p_{\sigma(m)}}$ and the finite $W$-algebra
$U(\g^{\s_\sigma},e)$ decomposes as a tensor product
$$
U(\g^{\s_j},e) \iso U(\gl_{n-p_{\sigma(m)}},e_\sigma') \otimes
U(\g^\phi_{2p_{\sigma(m)}},e_\sigma),
$$
where $e_\sigma'$ and $e_\sigma$ denote the projections of $e$ into
$\gl_{n-p_{\sigma(m)}}$ and $\g^\phi_{2p_{\sigma(m)}}$ respectively.
As explained before \eqref{e:translevi}, we have that
$L^{\s_\sigma}_{\sigma}(B)$ is finite dimensional.  Also
$L^{\s_\sigma}_{\sigma}(B)$ is the tensor product of irreducible
highest weight modules for $U(\gl_{n-p_{\sigma(m)}},e_0)$ and
$U(\g^\phi_{2p_{\sigma(m)}},e_1)$.  The tensor factor that is a
$U(\g^\phi_{2p_{\sigma(m)}},e_j)$-module is the highest weight
module labelled by $B_{-m}^m$.  Therefore, we have that $c \cdot
B_{-m}^m$ is defined by \cite[Theorem 1.2]{Bro2} as above.

Next we describe the action of $\tilde C(e)$ on $\cX^+(P)$.
Let $i_1 < \dots < i_d$ be minimal such that $p_{i_1},\dots,p_{i_d}$
are the distinct parts of $\bp = (p_1^2 \ge p_2^2 \ge \dots \ge
p_r^2)$ that are $\phi$-odd. Then we can choose generators
$c_1,\dots,c_d$ for $\tilde C(e) \iso \Z_2^d$ corresponding to
$p_{i_1},\dots,p_{i_d}$.  A lift of the element $c_j$ in $H^e$ fixes
all basis vectors $e_k$, except those where  $\row(k) = \pm i_j$. If
$l$ is in the same column as $k$ with $\row(k) = i_j$ with $\row(l)
= -i_j$, then up to a sign the lift of $c_j$ exchanges $e_k$ and
$e_l$. Explicit formulas for the lift of $c_j$ can be found in
\cite[\S5.3]{BroG}, these can be deduced from the explicit
description of centralizers given in \cite[Section 3]{Ja}.

Let $j = 1,\dots,d$, below we give the action of $c_j$.
Let $B \in \cX_\phi^+(P)$ and let $\sigma \in W_m$ such that $B \in
\sTab^+_\phi(\sigma \cdot P)$. Let $\tau \in \bar S_m$ be the
permutation
$$
\tau = \bar s_m \bar s_{m-1} \dots \bar s_{|\sigma(i_j)|}.
$$
We consider $L_{\tau \sigma}(\tau \star B)$, which is isomorphic to
$L_\sigma(B)$ by Proposition \ref{P:tophalf}.
From the formula
for the lift of $c_j$ given in \cite[\S5.3]{BroG}, we see that $c_j$ is in the subgroup of
$\tilde G$ isomorphic to $\tilde G^\phi_{2p_{\tau\sigma(1)}}$ corresponding to the direct
summand $\g^\phi_{2p_{\tau\sigma(1)}}$ of $\g^{\s_{\tau\sigma}}$.
Therefore, by \cite[Lemma 3.15]{BroG} and \cite[Theorem 6.1]{Bro2} we have that
$$
c_j \cdot [L_{\tau\sigma}(\tau \star B)] =  [L_{\tau\sigma}(c \cdot
(\tau \star B))].
$$
This leads us to define
\begin{equation} \label{e:cjact}
c_j \cdot B = \tau^{-1} \star (c \cdot (\tau \star B)).
\end{equation}
Then by Proposition \ref{P:tophalf} we obtain.

\begin{Proposition} \label{P:clact}
In the notation given above we have
$$
c_j \cdot [L_{\sigma}(B)] = [L_{\sigma}(\tau^{-1} \star (c \cdot
(\tau \star B))].
$$
\end{Proposition}

The following is an immediate consequence of Proposition
\ref{P:clact}.

\begin{Corollary}
The operation of the elements of $\tilde C(e)$ on $\cX^+_\phi(P)$ is
a $\tilde C(e)$ group action.
\end{Corollary}

We refer the reader to \cite[\S5.3]{BroG} for
some examples of applications of the operators $c_j$.

\begin{Remark} \label{R:altcomp}
We chose $i_j$ to be minimal for definiteness. Let $i_j'$ be such
that $p_{i_j'} = p_{i_j}$.  Then there is a lift of $c_j$ which
acts in the way described above except with $i_j'$ in place of
$i_j$.  The arguments above all go through with $i_j'$ in place
of $i_j$, so we could define $\tau$ with $i_j'$ in place of $i_j$,
and obtain an alternative formula for the action of $c_j$ on
$\cX_\phi^+(P)$  to that in \eqref{e:cjact}.
\end{Remark}

\subsection{The classification of finite dimensional irreducible $U(\g,e)$-modules}

Now that we have described the component group action we can state
classification of finite dimensional irreducible $U(\g,e)$-modules
with integral central character from \cite{BroG}.

\begin{Theorem}[{\cite[Theorem 5.13]{BroG}}] \label{T:BroGmain}
Let $A \in \sTab^\leq_\phi(P)$.  Then the $U(\g,e)$-module $L(A)$ is
finite dimensional if and only if $A$ is $\tilde C(e)$-conjugate to
a table that is justified row equivalent to column strict.
\end{Theorem}

\begin{Remark}
In case all parts of $\bp$ have the same parity, then $P$ is
justified, thus there is a natural notion of $A \in
\sTab^{\le}_\phi(P)$ being row equivalent to column strict as an
s-table.  By \cite[Lemma 4.13]{BroG} this is equivalent to being row
equivalent to column strict in the not skew-symmetric sense.
\end{Remark}

\subsection{The restricted Weyl group}
In this subsection we explain how to change highest weight theories using
elements of the restricted Weyl group $\tilde W^e$ as in \S\ref{ss:rest}.

First we recall the structure of the restricted Weyl group $\tilde
W^e$ from \cite[Sections 4, 6 and 7]{BruG}.  For $i = 1,\dots,m$, we
let $\bar m_i$ be the multiplicity of $i$ in $\bp$.  Then $\tilde
W^e$ is the subgroup of $W_m$ consisting of permutations of $\{\pm
1,\dots, \pm m\}$ that permute numbers labelling rows in $P$ of
equal length; so $W^e \iso W_{\bar m_1} \times \dots \times W_{\bar
m_m}$. We note that in \cite{BruG} only the group $W^e$ is
considered, but it is straightforward to deduce our assertions for
$\tilde W^e$. For $k = 1,\dots,m$ we let $r_k = (k,-k) \in \tilde
W^e$.

Recall the subgroups $Z^e$ and $(\tilde W^e)^\circ$ of $\tilde W^e$ from \S\ref{ss:rest}.
As explained in \cite[Section 4]{BruG}, we have $Z^e$ is isomorphic
to $\Z_2^d$, where as before $d$ is the number of $\phi$-odd parts of $\bp$.
Let $i_1 < \dots < i_d$ be as in \S\ref{ss:comp}, then one can easily
calculate that $Z^e$ is generated by the elements $r_{i_j} \in \tilde W^e$.
Also we have that $(\tilde W^e)^\circ$ is the subgroup $W'_{\bar m_1} \times \dots \times W'_{\bar m_m}$
where $W'_{\bar m_i}$ is the
Weyl group of type $D_{\bar m_i}$ if $i$ is $\phi$-odd and $W'_{\bar m_i} = W_{\bar m_i}$ if
$i$ is $\phi$-even.

Let $k \in \{1,\dots,m\}$.  We can write $r_k = z_kv_k$, where $z_k
\in Z^e$ and $v_k \in (\tilde W^e)^\circ$ as in \S\ref{ss:rest}.  If
$p_k$ is $\phi$-odd, then we let $j$ be such that $p_k = p_{i_j}$
and we see that $z_k = r_{i_j}$.  Further, recalling the maps
$\iota$ and $\kappa$ from \S\ref{ss:rest}, we have that $\iota(z_k)
= \kappa(c_j)$.  If $p_k$ is $\phi$-even, then we see that $z_k =
1$.

All the assertions above can be verified with the explicit
descriptions of centralizers given in \cite[Section 3]{Ja}.

Let $\sigma \in W_m$ and $B \in \sTab^+_\phi(\sigma \cdot P)$. By
applying Proposition \ref{P:restchange}, where we consider $r_k \in
\tilde W^e$ acting on $L_\sigma(B)$, we obtain
\begin{equation} \label{e:restactclass}
[L_\sigma(B)] = \begin{cases} [L_{r_{\sigma^{-1}(k)} \sigma}(c_j \cdot B)] & \text{if $p_k$ is $\phi$-odd, where $j$ is such that $p_k = p_{i_j}$}; \\
[L_{r_{\sigma^{-1}(k)}\sigma}(B)] & \text{if $p_k$ is $\phi$-even}.
\end{cases}
\end{equation}

\subsection{Changing highest weight theories for $U(\g,e)$}
\label{ss:change} We are now in a position to explain how to change
highest weight theories in general. To do this we extend the action
of $\bar S_m$ on $\cX_\phi^+(P)$ to an action of $W_m$.  The
important step in doing this is to define the action of $r = (m,-m)
\in W_m$.  This can be done in terms of the restricted Weyl group as
in the previous subsection.

Let $\sigma \in W_m$ and let $B \in \sTab_\phi^+(\sigma \cdot P)$.
By \eqref{e:restactclass} for $k = \sigma(m)$ we have
$$
[L_\sigma(B)] = \begin{cases}  L_{r\sigma}(c_j \cdot B) & \text{if $p_k$ is $\phi$-odd, where $j$ is such that $p_k = p_{i_j}$;} \\
L_{r\sigma}(B) & \text{if $p_k$ is $\phi$-even}.
\end{cases}
$$
Using Remark~\ref{R:altcomp} we see that in both cases this says
that
$$
[L_\sigma(B)] = [L_{r\sigma}(c \cdot B)],
$$
where the operation of $c$ is defined in \S\ref{ss:comp}.  Thus we define
the star action of $r$ on $\cX_\phi^+$ by
$$
r \star B = c \cdot B.
$$

We can now extend the $\star$-actions of $\bar S_m$ and $r$ on
$\cX^+$ to a $\star$-action of $W_m$ similarly to in the type $A$ case
as in \eqref{e:extend}. Then we get the following analogue of
Theorem \ref{T:hwA}.

\begin{Theorem} \label{T:changehwevenmult} $ $
\begin{enumerate}
\item[(i)]  The $\star$-action of $W_m$ on $\cX^+_\phi(P)$ is well defined.
\item[(ii)]
Let $\sigma,\tau \in W_m$, and $\bar B \in \sTab^+_\phi(\sigma \cdot
F)$, $\bar B' \in \sTab^+_\phi(\tau\sigma \cdot P)$. Then
$L_\sigma(\bar B) \iso L_{\tau\sigma}(\bar B')$ if and only if $B' =
\tau \star \bar B$. \end{enumerate}
\end{Theorem}

We demonstrate Theorem \ref{T:changehwevenmult} with an example for
$\g = \sp_{10}$.  We take $A$ as in \eqref{e:stableex} and $\tau =
(1,-2)(2,-1) = r \bar s_1 r$. To calculate $\tau \star A$ we first
calculate $r \star A$.  Since the length of row $1$ of $A$ is
$\phi$-even, we see that $r \star A = A$.  Next we calculate $\bar
s_1 \star A$ using the row swapping operation and we get
$$
    \bar s_1 \star A =
    \begin{array}{c}
\begin{picture}(60,80)
\put(0,0){\line(1,0){60}} \put(0,20){\line(1,0){60}}
\put(10,40){\line(1,0){40}} \put(0,60){\line(1,0){60}}
\put(0,80){\line(1,0){60}} \put(0,0){\line(0,1){20}}
\put(20,0){\line(0,1){20}} \put(40,0){\line(0,1){20}}
\put(60,0){\line(0,1){20}} \put(10,20){\line(0,1){40}}
\put(30,20){\line(0,1){40}} \put(50,20){\line(0,1){40}}
\put(0,60){\line(0,1){20}} \put(20,60){\line(0,1){20}}
\put(40,60){\line(0,1){20}} \put(60,60){\line(0,1){20}}
\put(30,40){\circle*{3}} \put(8,70){\makebox(0,0){{-3}}}
\put(28,70){\makebox(0,0){{-2}}} \put(50,70){\makebox(0,0){{7}}}
\put(20,50){\makebox(0,0){{1}}} \put(40,50){\makebox(0,0){{4}}}
\put(18,30){\makebox(0,0){{-4}}} \put(38,30){\makebox(0,0){{-1}}}
\put(8,10){\makebox(0,0){{-7}}} \put(30,10){\makebox(0,0){{2}}}
\put(50,10){\makebox(0,0){{3}}}
\end{picture}
\end{array}
$$
To finish off we have to apply $r \star$ to $\bar s_1 \star A$, which
means applying the operation of $c$ and gives
$$
    \tau \star A =
    \begin{array}{c}
\begin{picture}(60,80)
\put(0,0){\line(1,0){60}} \put(0,20){\line(1,0){60}}
\put(10,40){\line(1,0){40}} \put(0,60){\line(1,0){60}}
\put(0,80){\line(1,0){60}} \put(0,0){\line(0,1){20}}
\put(20,0){\line(0,1){20}} \put(40,0){\line(0,1){20}}
\put(60,0){\line(0,1){20}} \put(10,20){\line(0,1){40}}
\put(30,20){\line(0,1){40}} \put(50,20){\line(0,1){40}}
\put(0,60){\line(0,1){20}} \put(20,60){\line(0,1){20}}
\put(40,60){\line(0,1){20}} \put(60,60){\line(0,1){20}}
\put(30,40){\circle*{3}} \put(8,70){\makebox(0,0){{-3}}}
\put(28,70){\makebox(0,0){{-2}}} \put(50,70){\makebox(0,0){{7}}}
\put(18,50){\makebox(0,0){{-4}}} \put(40,50){\makebox(0,0){{1}}}
\put(18,30){\makebox(0,0){{-1}}} \put(40,30){\makebox(0,0){{4}}}
\put(8,10){\makebox(0,0){{-7}}} \put(30,10){\makebox(0,0){{2}}}
\put(50,10){\makebox(0,0){{3}}}
\end{picture}
\end{array}
$$

\begin{Remark}
In case all parts of $\bp$ are equal we say that $e$ is a
rectangular nilpotent element. In this case classification of finite
dimensional $U(\g,e)$-modules given in \cite{Bro2} does not have the
restriction to integral central characters.    If $\bp$ has an even
number of parts it is easy to see that Theorem
\ref{T:changehwevenmult} holds in this case without the restriction
to integral central characters. We note that in this case the action
of each $\bar s_i$ is trivial, as explained in Remark \ref{R:restA},
and the action of $r$ is given by $c$.  This is explained by the
fact that all possible choices of parabolic subalgebras can be
attained using the action of $\tilde W^e$.

When $\bp$ has an odd number of parts, it is still the case that all possible choices of
parabolic subalgebras can be attained using the action of $\tilde W^e$.  Therefore,
as the component group is trivial in this case, we see that all changes of highest weight
theory are given by a trivial action on s-tables.
\end{Remark}


\begin{thebibliography}{CPS33}

\bibitem[Br]{Bro2} J.~Brown,
{\em Representation theory of rectangular finite $W$-algebras},
J.\ Alg.\ to appear (2011).

\bibitem[BroG]{BroG} J.~Brown and S.~M.~Goodwin
{\em Finite dimensional irreducible representations of
finite $W$-algebras associated to even
multiplicity nilpotent orbits in classical Lie
algebras},
preprint, arXiv:1009.3869 (2010).

\bibitem[BruG]{BruG}
J.~Brundan and S.~M.~Goodwin, {\em Good grading polytopes}, Proc.\
London Math.\ Soc.\ {\bf 94} (2007), 155--180.

\bibitem[BGK]{BGK} J.~Brundan, S.~M.~Goodwin and A.~Kleshchev,
{\em Highest weight theory for finite $W$-algebras}, Internat.\
Math.\ Res.\ Notices, {\bf 15} (2008), Art.\ ID rnn051.

\bibitem[BK]{BK}
J.~Brundan and A.~Kleshchev, {\em Representations of
shifted Yangians and finite $W$-algebras}, Mem.\ Amer.\ Math.\ Soc.
{\bf 196} (2008).

\bibitem[EK]{EK}
A.~Elashvili and V.~Kac,  Classification of good gradings of simple
Lie algebras, in {\em Lie groups and invariant theory}\,
(E.~B.~Vinberg ed.), pp. 85--104, Amer.\ Math.\ Soc.\ Transl.\ {\bf
213}, AMS, 2005.

\bibitem[Fu]{F}
W.~Fulton, {\em Young tableaux}, LMS, 1997.

\bibitem[Ja]{Ja}
J. C. Jantzen, {\em Nilpotent orbits in representation theory},
Progress in Math., vol. 228, Birkh\"auser, 2004.

\bibitem[Jo1]{Jo1}
    A.~Joseph,
    {\em Sur la classification des id\'eaux primitifs dans l'alg\`ebre
    enveloppante de $\sl(n + 1, \C)$},
    C.~R.~Acad.~Sci.~Paris {\bf 287} (1978), A303--A306.

\bibitem[Jo2]{Jo2}
\bysame,
{\em Towards the Jantzen conjecture III}, Compositio
Math.\ {\bf 41} (1981), 23--30.

\bibitem[Ko]{Ko}
B.~Kostant, {\em On Whittaker modules and representation theory},
Invent.\ Math.\ {\bf 48} (1978), 101--184.

\bibitem[Lo1]{Lo1}
I.~Losev, {\em Quantized symplectic actions and $W$-algebras}, J.\
Amer.\ Math.\ Soc.\  {\bf 23}  (2010),  no.\ 1, 35--59.

\bibitem[Lo2]{Lo2} \bysame,
{\em Finite dimensional representations of $W$-algebras}, preprint,
arXiv:0807.1023 (2008).

\bibitem[Lo3]{Lo3} \bysame,
{\em On the structure of the category $\mathcal O$ for
$W$-algebras}, preprint, arXiv:0812.1584 (2008).

\bibitem[Lo4]{Lo4} \bysame,
{\em Finite $W$-algebras}, preprint, arXiv:1003.5811, (2010).

\bibitem[MS]{MS}
D. Mili\v c\'ic and W. Soergel,
{\em The composition series of modules induced from Whittaker modules},
Comment.\ Math.\ Helv.\ {\bf 72} (1997), 503--520.

\bibitem[Pr1]{Pr1}
A.~Premet, {\em Special transverse slices and their enveloping
algebras}, Adv.\ in Math.\ {\bf 170} (2002), 1--55.

\bibitem[Pr2]{Pr2} \bysame,
{\em Enveloping algebras of Slodowy slices and the Joseph ideal},  J.\
Eur.\ Math.\ Soc.\ {\bf 9} (2007), 487--543.

\bibitem[Sk]{Sk} S.~Skryabin, {\em A category equivalence},
appendix to \cite{Pr1}.

\end{thebibliography}
\end{document}